\documentclass[10pt,reqno]{amsart}
\usepackage{amsmath} 
\usepackage{amssymb}
\usepackage{dsfont}
\usepackage[dvips,draft,final]{graphics}
\usepackage[T1]{fontenc}
\usepackage{lmodern}
\usepackage{fancyhdr}
\usepackage{url}
\usepackage{enumerate}
\usepackage{color}
\usepackage{graphicx}
\usepackage{mathrsfs}

\newtheorem{theorem}{Theorem}[section]
\newtheorem{proposition}[theorem]{Proposition}
\newtheorem{lemma}[theorem]{Lemma}

\theoremstyle{definition}

\newcommand{\bel}{\begin{equation} \label}
\newcommand{\ee}{\end{equation}}

\newcommand{\pd}{\partial}

\newcommand{\R}{{\mathbb R}}

\def\epsilon{\varepsilon}
\def\phi {\varphi}

\def\beq{\begin{equation}}
\def\eeq{\end{equation}}
\renewcommand{\leq}{\leqslant}
\renewcommand{\geq}{\geqslant}
\newcommand{\bea}{\begin{eqnarray}}
\newcommand{\eea}{\end{eqnarray}}
\newcommand{\beas}{\begin{eqnarray*}}
\newcommand{\eeas}{\end{eqnarray*}}

{

\providecommand{\abs}[1]{\left\lvert#1\right\rvert}
\providecommand{\norm}[1]{\left\lVert#1\right\rVert}

\title[Stable recovery of non-compactly supported  electromagnetic potentials]{Stable recovery  of  non-compactly supported  electromagnetic potentials in unbounded domain}
\author{Yavar Kian}
\address{Aix Marseille Univ, Universit\'e de Toulon, CNRS, CPT, Marseille, France.}
\email{yavar.kian@univ-amu.fr}

\author{Yosra Soussi}
\address{Universit\'e de Tunis El Manar, Ecole d'in\'enieur de Tunis, ENIT-LMSIN, B.P. 37, 1002 Tunis, Tunisia\\
Aix Marseille Univ, Universit\'e de Toulon, CNRS, CPT, Marseille, France.}
\email{yosra.soussi@enit.utm.tn}

\begin{document}
\begin{abstract} We consider the inverse problem of determining an electromagnetic potential appearing in an infinite cylindrical domain from boundary measurements. More precisely, we prove the stable recovery of some general class of magnetic field and electric potential from boundary measurements. Assuming some knowledge of the unknown coefficients close to the boundary, we obtain also some results of stable recovery with measurements restricted to some portion of the boundary. Our approach combines construction of complex geometric optics solutions and Carleman estimates suitably designed for our stability results stated in an unbounded domain.

\medskip
\noindent
{\bf Keywords :} Inverse problems, elliptic equations, electromagnetic potential, Carleman estimate, unbounded domain, closed  waveguide,  partial data.

\medskip
\noindent
{\bf Mathematics subject classification 2010 :} 35R30, 35J15.
\end{abstract}
\maketitle


\section{Introduction}
\label{sec-intro}
\setcounter{equation}{0}
\subsection{Statement of the problem}
Let  $\Omega$ be an open set of $\R^3$ corresponding  to a closed waveguide. More precisely,  we assume that there exists $\omega$  a $\mathcal C^3$  bounded, open and simply connected set of $\mathbb{R}^2$ such that $\Omega=\omega\times\R$.
 For $A\in W^{1,\infty}(\Omega)^3$, we define the magnetic Laplacian $\Delta_A$ given by
$$\Delta_A=\Delta+2iA\cdot\nabla+i\textrm{div}(A)-|A|^2.$$
For $q \in L^\infty(\Omega)$ such that $0$ is not in the spectrum of of the operator $-\Delta_A+ q$ acting on $L^2(\Omega)$ with Dirichlet boundary condition,  we can introduce the  boundary value problem
\bel{eq1}
\left\{ 
\begin{array}{rcll} 
(-\Delta_A + q) u & = & 0, & \mbox{in}\ \Omega,\\ 
u & = & f ,& \mbox{on}\ \Gamma : = \pd \Omega.
\end{array}
\right.
\ee
Recall that $\Gamma=  \pd \omega\times\R $ and that the outward unit normal vector $\nu$  to $\Gamma$ takes the form
$$ \nu(x',x_3)=(\nu'(x'),0),\ x=(x',x_3)\in\Gamma, $$
where $\nu'$ is the outward unit normal vector of $\pd \omega$. In the present paper we consider the simultaneous stable recovery of the magnetic field associated with $A$ and the electric potential $q$ from the full and  partial knowledge of the  Dirichlet-to-Neumann (DN in short) map
\begin{equation}\label{a1}
\begin{array}{lll}
\Lambda_{A,q} : & H^{\frac{3}{2}} (\partial \Omega) &\longrightarrow \quad H^{\frac{1}{2}} (\partial \Omega)\\
&\quad f &\longmapsto {(\pd_\nu+iA\cdot\nu) u}_{|\partial\Omega} ,
\end{array}
\end{equation}
where $\pd_\nu $ is the normal derivative. Let $\Gamma_0 \subset \partial \omega $ be an arbitrary open set. The restriction $\Lambda'_{A,q}$ of $\Lambda_{A,q}$ on $\Gamma_0 \times \R$ is defined by
\begin{equation}\label{DNp}
\begin{array}{lll}
\Lambda'_{A,q} : & H^{\frac{3}{2}} (\partial \Omega) &\longrightarrow \quad H^{\frac{1}{2}} (\Gamma_0 \times \R)\\
&\quad f &\longmapsto {(\pd_\nu+iA\cdot\nu) u}_{ \vert{\Gamma_0 \times \R}}.
\end{array}
\end{equation}

\subsection{Motivations}
The problem addressed in this article  is connected with the   electrical impedance tomography (EIT in short) method as well as its applications in different scientific areas (e.g. medical imaging, geophysical prospection...). We refer to \cite{Uh} for a review of this problem. Our formulation of this problem in an unbounded closed waveguide can be associated with problems of transmission to long distance or transmission through  structures having important ratio length-to-diameter (e.g. nanostructures). The main objective of our study  is to determine in a stable way  an  electromagnetic impurity  perturbing the guided propagation (see for instance in \cite{CL,KBF}). 
 
\subsection{Known results}
There have been many works so far devoted to the study of the Calder\'on problem initially stated in \cite{Ca}. The first positive answer to this problem can be found in \cite{SU} where the authors used an approach based on the construction of complex geometric optics (CGO in short) solutions. We refer also to \cite{Ch,Ha,Ki6} for some alternative constructions of CGO solutions. Motivated by this result, many authors investigated several aspects of this problem.  One of the first results devoted to the recovery of electromagnetic potentials can be found in \cite{Suu}. Here the authors stated a uniqueness result under a smallness assumption of the associated magnetic field. This smallness assumption has been removed by \cite{NSU} for smooth coefficients and improved in terms of regularity by \cite{KU}. Since then, in \cite{T}, the author considered  magnetic potentials lying in $\mathcal C^1$, \cite{Sa1} treated the case of    magnetic potentials lying in a Dini class and \cite{KU} considered this problem with bounded electromagnetic potentials. One of the first results of stability for this problem can be found in \cite{Tz}
and, without being exhaustive, we refer to \cite{B,CP,Pot2,Pot1} for some recent improvements of such results and to the works of \cite{Al,CDR2,CKS4,KU1} for the stable recovery of several classes of coefficients appearing in an elliptic equation.

All the above mentioned results are stated in a bounded domain. There have been only few works devoted to the recovery of coefficients for elliptic equations in an unbounded domain. Among these results several works have been devoted to the  recovery of coefficients of an elliptic equation in a Slab (see e.g. \cite{CM,KLU,LU}) and we refer to the works \cite{CKS2,CKS3}  for the recovery of periodic coefficients in an infinite waveguide. As far as we know, the first results dealing with the unique recovery of general class of non-compactly supported  and non-periodic coefficients, appearing in an unbounded cylindrical domain, can be found in \cite{Ki4,Ki5}. More recently, in \cite{So} the author proved the stable recovery of an electric potential similar to the class of coefficients under consideration in \cite{Ki4}. To the best of our knowledge, the results of \cite{So} correspond to the first proof of stable recovery of coefficients similar to those considered by \cite{Ki4} from full and partial data.
We mention also the works \cite{BKS,BKS1,CKS, CS, KKS,Ki1,KPS1, KPS2} dealing with similar problems in a different class of PDEs.

\subsection{Statement of the main results}

 Taking into account the well known obstruction to the recovery of the electromagnetic potentials (see e.g. \cite[Section 1.4]{Ki5}), we study the stable recovery of the magnetic field and the electric potential appearing in \eqref{eq1}. More precisely, for $A=(a_1,a_2,a_3)$,  we consider the recovery of the magnetic field corresponding to the 2-form valued distribution $dA$ defined by
$$dA:=\underset{1\leq j<k\leq 3}{\sum} (\partial_{x_j}a_k-\partial_{x_k}a_j)dx_j\wedge dx_k$$
and the electric potential $q$. In our first result, we prove the stable recovery of the magnetic field.

\begin{theorem}
\label{t1} 
For $j=1,2$, let $A_j\in W^{2,1}(\Omega)^3\cap W^{2,\infty}(\Omega)^3$ satisfy the condition 
\bel{t1a} \partial_x^\alpha A_1(x)=\partial_x^\alpha A_2(x),\quad x\in\partial\Omega,\ \alpha\in\mathbb N^3,\ |\alpha|\leq 1\ee
and assume that $0$ is not in the spectrum of the operator $-\Delta_{A_j}+q_j$ acting in $L^2(\Omega)$ with Dirichlet boundary condition.
Assume also that there exist $M>0$, $s\in(0,1/2)$ and $f\in L^{\frac{5}{3}}(\R_+;\R_+)$ a  decreasing function such that the following conditions are fulfilled
\bel{t1b} \begin{aligned}\int_{\Omega} \left\langle x_3\right\rangle^s|A_1(x)-A_2(x)|dx+\norm{r\mapsto r^{\frac{3}{5}}f(r)}_{L^{\frac{5}{3}}(1,+\infty)}&\leq M,\\
\sum_{j=1}^2[\norm{A_j}_{W^{2,\infty}(\Omega)^3}+\norm{A_j}_{H^2(\Omega)^3}+\norm{q_j}_{L^\infty(\Omega)}] &\leq M,   \end{aligned}\ee
\bel{t1c} | A_1(x)- A_2(x)|\leq f(|x|),\quad x\in\Omega.    \ee
Then there exist $C>0$ depending only on $\Omega$, $s$, $f$ and $M$ and $s_1\in(0,1)$ depending only on $s$ such that the following estimate
\bel{t1d}\norm{dA_1-dA_2}_{L^2(\Omega)}\leq C\ln\left(3+\norm{\Lambda_{A_1,q_1}-\Lambda_{A_2,q_2}}_{\mathcal B(H^{\frac{3}{2}}(\partial\Omega),L^2(\partial\Omega))}^{-1}\right)^{-s_1}\ee
holds true.
\end{theorem}

Assuming that the divergence of the magnetic potential under consideration is known, we prove also the stable recovery of the electric potential.

\begin{theorem}\label{t2}
Let the condition of Theorem \ref{t1} and conditions \eqref{t1a}-\eqref{t1c} be fulfilled. Assume also that 
\bel{t2a} \textrm{div}(A_1)=\textrm{div}(A_2).\ee
Moreover, let $q_j\in H^1(\Omega)\cap L^2(\Omega)$, $j=1,2$, satisfy the following condition
\bel{t2aa} \int_{\Omega} \left\langle x_3\right\rangle^s|q_1(x)-q_2(x)|dx+ \norm{q_1-q_2}_{H^1(\Omega)^3}\leq M,\ee
with $s\in(0,1)$.
Then there exists a constant $s_2>0$ depending only on $s$ such that the following estimate
\bel{t2b}\norm{q_1-q_2}_{L^2(\Omega)}\leq C\ln\left[\ln\left(e^3+\norm{\Lambda_{A_1,q_1}-\Lambda_{A_2,q_2}}_{\mathcal B(H^{\frac{3}{2}}(\partial\Omega),L^2(\partial\Omega))}^{-1}\right)\right]^{-s_2}\ee
holds true, with $C>0$ depending only on $\Omega$, $s$ and $M$.
\end{theorem}

Now, we give two partial data results with restriction of the measurements to an arbitrary subset of the boundary. The statement of these results requires some definitions and assumptions that we need to recall first. Let $\mathcal{W}_0 \subset \omega $ be an arbitrary neighborhood of the boundary $\partial \omega$ such that $\partial \mathcal{W}_0 = \partial \omega \cup \Gamma^{\sharp} $ with $\partial \omega \cap \Gamma^{\sharp} = \varnothing $. We assume that $\Gamma^\sharp$ is $\mathcal{C}^2$. Let $\Gamma_0 \subset \partial \omega \subset \partial \mathcal{W}_0$ be an arbitrary (not empty) open set of $\partial \omega$ and let $\mathcal{O}_0 = \mathcal{W}_0 \times \R$.
For a given $M>0$, we introduce the admissible sets of coefficients
$$\mathcal{A}(M, A_0 ,\mathcal{O}_0) = \lbrace A \in  \mathcal{C}^2(\overline{\Omega}, \R^3); \, \Vert A \Vert_{\mathcal{C}^2(\overline{\Omega})} \leqslant M \text{ and } A = A_0 \text{ in } \mathcal{O}_0  \rbrace ,$$
$$\mathcal{Q}(M, q_0 ,\mathcal{O}_0) = \lbrace q \in  L^{\infty}(\overline{\Omega}, \R^3); \, \Vert q \Vert_{L^{\infty}(\overline{\Omega})} \leqslant M \text{ and } q = q_0 \text{ in } \mathcal{O}_0  \rbrace .$$
\begin{theorem}
\label{st3} 
For $j=1,2$, let $q_j\in \mathcal{Q}(M, q_0 ,\mathcal{O}_0)$ and let $A_j\in \mathcal{A}(M, A_0 ,\mathcal{O}_0)$ satisfy the conditions of Theorem \ref{t1}. 
Then there exist $C>0$ depending only on $\Omega$, $s$ and $M$ and $s_1$ depending only on $s$ such that the following estimate
\bel{res3}\norm{dA_1-dA_2}_{L^2(\Omega)}\leq C\ln\left(3+\norm{\Lambda'_{A_1,q_1}-\Lambda'_{A_2,q_2}}_{\mathcal{B}(H^{\frac{3}{2}} (\partial \Omega) , H^{\frac{1}{2}} (\Gamma_0 \times R ))}^{-1}\right)^{-s_1}\ee
holds true.
\end{theorem}
\begin{theorem}\label{st4}
For $j=1,2$, let $A_j\in \mathcal{A}(M, A_0 ,\mathcal{O}_0)$ and $q_j\in \mathcal{Q}(M, q_0 ,\mathcal{O}_0)$ satisfy the conditions of Theorem \ref{t2}. Then there exists a constant $s_2>0$ depending only on $s$ such that the following estimate
\bel{res4}\norm{q_1-q_2}_{L^2(\Omega)}\leq C\ln\left[\ln\left(e^3+\norm{\Lambda'_{A_1,q_1}-\Lambda'_{A_2,q_2}}_{\mathcal{B}(H^{\frac{3}{2}} (\partial \Omega) , H^{\frac{1}{2}} (\Gamma_0 \times R ))}^{-1}\right)\right]^{-s_2}\ee
holds true, with $C>0$ depending only on $\Omega$, $s$ and $M$.
\end{theorem}

To the best of our knowledge Theorem \ref{t1} and \ref{t2} correspond to the first results of stable recovery of the magnetic field and the electric potential associated with non-compactly supported electromagnetic potentials from boundary measurements. Indeed, while the uniqueness of this problem can be found in \cite{Ki5}, the stability issue for this problem has not been treated so far. We mention that in contrast to bounded domains, in unbounded domains the transition from a uniqueness result  to a stability estimate requires some careful analysis and one has to deal with several difficulties strongly related to the lost of compactness of the closure of the domain. For instance, in order to obtain such stability estimates we need to impose the extra assumptions \eqref{t1b}-\eqref{t1c} and \eqref{t2aa} to the electromagnetic potentials under consideration. Roughly speaking conditions \eqref{t1b}-\eqref{t1c} and \eqref{t2aa} claim that the difference of the electromagnetic potentials under consideration admit some decay at infinity. It is not clear how one can get a stability result associated with the results of \cite{Ki5} without assuming such extra assumptions.

In the spirit of \cite{B}, in Theorem \ref{st3} we treat the stable recovery of electromagnetic potentials that are known close to the boundary from some suitable sub-boundary of $\partial\Omega$. Our approach requires both results of Theorem \ref{t1} and \ref{t2} and some extension of the arguments of \cite{B} to an unbounded cylindrical domain. This includes a weak unique continuation result, stated in Lemma \ref{UCP}, that we derive for unbounded cylindrical domains.  

In contrast to similar results stated in bounded domains (see e.g. \cite{Tz}), in Theorem \ref{t2} we assume the knowledge of the divergence of the magnetic potentials under consideration in order to prove the recovery of the electric potentials. This is related to the fact that, in contrast to bounded domains, it is not clear how one can exploit  the gauge invariance associated with our problem, introduced in \cite[Section 1.4]{Ki5}, for showing the stable recovery of the electric potential without assuming   the knowledge of the divergence of the magnetic potential. 

\subsection{Outlines}
This paper is organized as follows. In Section 2 we introduce some class of CGO solutions, suitably designed for our problem, that we build  by mean of  Carleman estimates. In Section 3, we complete the proof of the results with full boundary measurements stated in Theorem \ref{t1} and \ref{t2}. Section 4 will be devoted to the results stated in Theorem \ref{st3} using partial boundary measurements. Finally, in the appendix we prove several intermediate results including an interpolation result, a Carleman estimate and a weak unique continuation property.

\section{CGO solutions }
\label{sec2}
In this section we introduce a class of CGO solutions suitably designed for our problem stated in an unbounded domain for magnetic Schr\"odinger equations. More precisely, we consider  CGO solutions  $u_j\in H^2(\Omega_1)$, $j=1,2$, satisfying
$\Delta_{A_1} u_1+q_1u_1=0$, $\Delta_{A_2} u_2+q_2u_2=0$  in $\Omega$ for $A_j\in W^{1,\infty}(\Omega)^3\cap H^2(\Omega)^3$ and $q_j\in L^\infty(\Omega)$ satisfying \eqref{t1a}-\eqref{t1c}. In a similar way to \cite{Ki4, Ki5}, we consider first $\theta\in\mathbb S^{1}:=\{y\in\R^2:\ |y|=1\}$, $\xi'\in\theta^\bot\setminus\{0\}$, with $\theta^\bot:=\{y\in\R^2:\ y\cdot\theta=0\}$, $\xi:=(\xi',\xi_3)\in \R^3$, with  $\xi_3\neq0$. Then, we define $\eta\in\mathbb S^2:=\{y\in\R^3:\ |y|=1\}$  by
$$\eta=\frac{(\xi',-\frac{|\xi'|^2}{\xi_3})}{\sqrt{|\xi'|^2+\frac{|\xi'|^4}{\xi_3^2}}}.$$ 
Clearly, we have
\bel{orth}\eta\cdot\xi=(\theta,0)\cdot\xi=(\theta,0)\cdot\eta=0.\ee
We fix also $\psi\in\mathcal C^\infty_0((-2,2);[0,1])$ satisfying $\psi=1$ on $[-1,1]$  and, for $\rho>1$, we introduce solutions $u_j\in H^2(\Omega)$ of
 $\Delta_{A_1} u_1+q_1u_1=0$, $\Delta_{A_2} u_2+q_2u_2=0$ in $\Omega$ of the form
\bel{CGO1}u_1(x',x_3)=e^{\rho \theta\cdot x'}\left(\psi\left(\rho^{-\frac{1}{4}}x_3\right)b_1e^{i\rho x\cdot\eta-i\xi\cdot x}+w_{1,\rho}(x',x_3)\right),\quad x'\in\omega,\ x_3\in\R,\ee
\bel{CGO2}u_2(x',x_3)=e^{-\rho \theta\cdot x'}\left(\psi\left(\rho^{-\frac{1}{4}}x_3\right)b_2 e^{i\rho x\cdot\eta}+w_{2,\rho}(x',x_3)\right),\quad x'\in\omega,\ x_3\in\R.\ee
Here, for $j=1,2$, $b_j\in W^{2,\infty}(\Omega)$ and the remainder term $w_{j,\rho}\in H^2(\Omega)$ satisfies the decay property 
\bel{RCGO} \begin{aligned}&\rho^{-1}\norm{w_{1,\rho}}_{H^2(\Omega)}+\norm{w_{1,\rho}}_{H^1(\Omega)}+\rho\norm{w_{1,\rho}}_{L^2(\Omega)}\leq C(|\xi|^2+1)\left(1+\frac{|\xi'|}{|\xi_3|}\right)\rho^{\frac{7}{8}},\\
&\rho^{-1}\norm{w_{2,\rho}}_{H^2(\Omega)}+\norm{w_{2,\rho}}_{H^1(\Omega)}+\rho\norm{w_{2,\rho}}_{L^2(\Omega)}\leq C\left(1+\frac{|\xi'|}{|\xi_3|}\right)\rho^{\frac{7}{8}},\end{aligned}\ee
with $C>0$ depending on $\Omega$ and $\norm{A_j}_{ W^{1,\infty}(\Omega)^3}+\norm{q_j}_{L^\infty(\Omega)}$, $j=1,2$. 
 We summarize this construction as follows.
\begin{theorem}\label{t4} For $j=1,2$ and  for all $\rho>\rho_2$, with $\rho_2$ the constant of Proposition \ref{p2}, the equations  $\Delta_{A_1} u_1+q_1u_1=0$ and $\Delta_{A_2} u_2+q_2u_2=0$  admit  a solution $u_j\in H^2(\Omega)$ of the form \eqref{CGO1}-\eqref{CGO2} with $w_{j,\rho}$ satisfying the decay property \eqref{RCGO}.  

\end{theorem}
\subsection{Principal parts of the CGO}
In this section, we consider $A_j\in W^{2,\infty}(\Omega)^3$, $j=1,2$ satisfying \eqref{t1a}. From now on, for all $r>0$, we define $B_r:=\{x\in\R^{3}:\ |x|<r\}$ and $B_r':=\{x'\in\R^{2}:\ |x'|<r\}$.
In order to define $b_j$, $j=1,2$, we start by introducing a suitable extension of the coefficients $A_j$, $j=1,2$.  Following \cite[Lemma 3.1.]{BKS1}, we define $\tilde{A}_j\in W^{2,\infty}(\R^3)^3$, $j=1,2$, and we fix $r_0>0$ such that
\bel{ext} \begin{aligned} \tilde{A}_j(x)=A_j(x),\quad x\in\Omega\\
\textrm{supp}(\tilde{A}_j)\subset B'_{r_0}\times\R\\
\tilde{A}_1(x)=\tilde{A}_2(x),\quad x\in\R^3\setminus\Omega\\
\norm{\tilde{A}_j}_{ W^{2,\infty}(\R^3)^3}\leq C(\norm{A_1}_{ W^{2,\infty}(\Omega)^3}+\norm{A_2}_{ W^{2,\infty}(\Omega)^3}),\end{aligned}\ee
with $C>0$ depending only on $\Omega$. Here $r_0$ will only depends on $\Omega$.

 Following \cite{Ki5}, we fix $\tilde{\theta}=(\theta,0)\in\R^3$ and  we define
\bel{conde1} \begin{aligned}&\Phi_1(x):=\frac{-i}{2\pi} \int_{\R^2} \frac{(\tilde{\theta}+i\eta)\cdot \tilde{A}_1(x-s_1\tilde{\theta}-s_2\eta)}{s_1+is_2}ds_1ds_2,\\
 &\Phi_2(x):=\frac{-i}{2\pi} \int_{\R^2} \frac{(-\tilde{\theta}+i\eta)\cdot \tilde{A}_2(x+s_1\tilde{\theta}-s_2\eta)}{s_1+is_2}ds_1ds_2.\end{aligned}\ee
According to \cite{Ki5}, one can check that $\Phi_j\in W^{2,\infty}(\Omega)$. Assuming that condition \eqref{t1c} is fulfilled, we can even prove the following estimates.

\begin{lemma} \label{l1} Assume that condition \eqref{t1c} is fulfilled. Then, for all $R\geq r_0$, there exists $C>0$ depending on $\Omega$ and $R$  such that the following estimates 
\bel{l1a} \norm{\Phi_1}_{W^{2,\infty}(B'_R\times \R)}+\norm{\Phi_2}_{W^{2,\infty}(B'_R\times \R)}\leq C(\norm{A_1}_{ W^{2,\infty}(\Omega)^3}+\norm{A_2}_{ W^{2,\infty}(\Omega)^3})\left(1+\frac{|\xi'|}{|\xi_3|}\right),\ee
\bel{l1b} \norm{\Phi_1+\overline{\Phi_2}}_{L^\infty(\R^3)}\leq C\left(\norm{A_1}_{L^\infty(\Omega)}+\norm{A_2}_{L^\infty(\Omega)}+\norm{r\mapsto rf(r)}_{L^{\frac{5}{3}}(1,+\infty)}\right),\ee
hold true.
\end{lemma}

\begin{proof} We will prove the estimate \eqref{l1a} only for $\Phi_1$, the proof for $\Phi_2$ being similar. For this purpose, we fix $R\geq r_0$. For $\alpha\in\mathbb N^n$, $|\alpha|\leq2$, we have
$$|\partial_x^\alpha\Phi_1(x)|\leq \frac{1}{2\pi} \int_{\R^2} \frac{|\partial_x^\alpha\tilde{A}_1(x-s_1\tilde{\theta}-s_2\eta)|}{|s_1+is_2|}ds_1ds_2.$$
On the other hand, using the fact that supp$(\tilde{A}_1)\subset B'_{r_0}\times\R$, one can check that, for all $x\in B'_R\times \R$, we get
$$|\partial_x^\alpha\tilde{A}_1(x-s_1\tilde{\theta}-s_2\eta)|=0,\quad |(s_1,s_2)|\geq \frac{2R}{|(\eta_1,\eta_2)|},\ x\in\Omega,$$
where $\eta=(\eta_1,\eta_2,\eta_3)$. It follows that
\bel{l1c}\begin{aligned}|\partial_x^\alpha\Phi_1(x)|&\leq \frac{1}{2\pi} \norm{\tilde{A}_1}_{ W^{2,\infty}(\R^3)^3}\left(\int_{|(s_1,s_2)|\leq \frac{2R}{|(\eta_1,\eta_2)|}} \frac{1}{|s_1+is_2|}ds_1ds_2\right)\\
\ &\leq C\norm{\tilde{A}_1}_{ W^{2,\infty}(\R^3)^3} |(\eta_1,\eta_2)|^{-1}.\end{aligned}\ee
Recalling that 
$$(\eta_1,\eta_2)=\frac{\xi'}{\sqrt{|\xi'|^2+\frac{|\xi'|^4}{\xi_3^2}}}$$
and applying \eqref{ext}, we deduce \eqref{l1a} from \eqref{l1c}. 

Now let us consider \eqref{l1b}. Let us first observe that, according to \eqref{ext}, we have
$${\Phi_1(x)+\overline{\Phi_2}(x)}=\frac{-i}{2\pi} \int_{\R^2} \frac{(\tilde{\theta}+i\eta)\cdot A(x-s_1\tilde{\theta}-s_2\eta)}{s_1+is_2}ds_1ds_2,$$
where $A=A_1-A_2$ is extended by zero to $\R^3$. Therefore, applying \eqref{t1c} and the fact that $f$ is a  decreasing function, we deduce that

$$\begin{aligned}&|\Phi_1(x)+\overline{\Phi_2}(x)|\\
&\leq C \left((\norm{A_1}_{L^\infty(\Omega)}+\norm{A_2}_{L^\infty(\Omega)})\int_{B'_1} \frac{1}{|s_1+is_2|}ds_1ds_2+\int_{\R^2\setminus B'_1} \frac{|f(|x-s_1\tilde{\theta}-s_2\eta|)|}{|s_1+is_2|}ds_1ds_2\right)\\
&\leq C \left(\norm{A_1}_{L^\infty(\Omega)}+\norm{A_2}_{L^\infty(\Omega)}+\left(\int_{\R^2\setminus B'_1} |f(|x-s_1\tilde{\theta}-s_2\eta|)|^{\frac{5}{3}}ds_1ds_2\right)^{\frac{3}{5}}\left(\int_{\R^2\setminus B'_1}|s_1+is_2|^{-\frac{5}{2}}ds_1ds_2\right)^{\frac{2}{5}}\right)\\
\ &\leq C\left(\norm{A_1}_{L^\infty(\Omega)}+\norm{A_2}_{L^\infty(\Omega)}+\left(\int_{\R^2} |f(|(x\cdot\xi)\frac{\xi}{|\xi|}+s_1\tilde{\theta}+s_2\eta|)|^{\frac{5}{3}}ds_1ds_2\right)^{\frac{3}{5}}\right)\\
\ &\leq C\left(\norm{A_1}_{L^\infty(\Omega)}+\norm{A_2}_{L^\infty(\Omega)}+\left(\int_{\R^2} |f(|s_1\tilde{\theta}+s_2\eta|)|^{\frac{5}{3}}ds_1ds_2\right)^{\frac{3}{5}}\right)\\
\ &\leq C\left(\norm{A_1}_{L^\infty(\Omega)}+\norm{A_2}_{L^\infty(\Omega)}+\left(\int_{\mathbb S^1}\int_1^{+\infty} r|f(r)|^{\frac{5}{3}}drd\omega_1\right)^{\frac{3}{5}}\right)\\
\ &\leq C\left(\norm{A_1}_{L^\infty(\Omega)}+\norm{A_2}_{L^\infty(\Omega)}+\norm{r\mapsto r^{\frac{3}{5}}f(r)}_{L^{\frac{5}{3}}(1,+\infty)}\right).\end{aligned}$$
This completes the proof of the lemma. \end{proof}

Fixing
\bel{conde2} b_1(x)=e^{\Phi_{1}(x)},\quad b_2(x)=e^{\Phi_{2}(x)},\ee
we obtain 
\bel{tt}(\tilde{\theta}+i\eta)\cdot \nabla b_1+i[(\tilde{\theta}+i\eta)\cdot \tilde{A}_1(x)]b_1=0,\quad (-\tilde{\theta}+i\eta)\cdot \nabla b_2+i[(-\tilde{\theta}+i\eta)\cdot \tilde{A}_2(x)]b_2=0,\quad x\in\R^3.\ee
Here, using  the fact that $\overline{\omega}\subset B'_{r_0}$, we obtain
\bel{cond31}\norm{b_j}_{W^{2,\infty}( B_{r_0+1}'\times\R)}\leq C(\norm{A_1}_{ W^{2,\infty}(\Omega)^3}+\norm{A_2}_{ W^{2,\infty}(\Omega)^3})\left(1+\frac{|\xi'|}{|\xi_3|}\right),\ j=1,2.\ee

Using these properties of the expressions $b_j$, $j=1,2$, we will complete the construction of the solutions $u_j$ of the form \eqref{CGO1}-\eqref{CGO2}. For this purpose, we will use some suitable Carleman estimates which will extend the one introduced in \cite{Ki5} (see also \cite{FKSU,ST}).

\subsection{General Carleman estimate}

Let us first introduce a weight function depending on two parameters $s,\rho\in(1,+\infty)$ with $\rho>s>1$. Following \cite[Section 2.1]{Ki5},   for $\theta\in\mathbb S^2$ we consider the perturbed weight
\bel{phi}\phi_{\pm,s}(x',x_3):=\pm \rho\theta\cdot x'-s{(x'\cdot\theta)^2\over 2},\quad x=(x',x_3)\in\Omega.\ee
We introduce also the weighted operator
\[ P_{A,q,\pm,s}:=e^{-\phi_{\pm,s}}(\Delta +2iA\cdot\nabla +q)e^{\phi_{\pm,s}}.\]
Then we can consider the following Carleman estimate.
\begin{proposition}\label{p1} \emph{(Proposition 2.1, \cite{Ki5})} Let $A\in L^\infty(\Omega)^3\cap L^\infty(\Omega)^3$ and $q\in L^\infty(\Omega)$. Then  there exist $s_1>1$ and, for $s>s_1$,  $\rho_1(s)$ such that for any $v\in H^2(\Omega)\cap H^1_0(\Omega_1)$ 
the estimate
\bel{p1a} \begin{aligned}&\rho\int_{\partial\omega_{\pm,\theta}\times\R} |\partial_\nu v|^2|\theta\cdot \nu| d\sigma(x)+s\rho^{-2}\int_{\Omega_1}|\Delta v|^2dx+ s\int_{\Omega_1}|\nabla v|^2dx+s\rho^2\int_{\Omega_1}|v|^2dx \\
&\leq C\left[\norm{P_{A,q,\pm,s}v}^2_{L^2(\Omega_1)}+\rho\int_{\partial\omega_{\mp,\theta}\times\R} |\partial_\nu v|^2|\theta\cdot \nu| d\sigma(x)\right]\end{aligned}\ee
holds true for $s>s_1$, $\rho\geq \rho_1(s)$  with $C$  depending only on  $\Omega$ and $ \norm{q}_{L^\infty(\Omega)}+\norm{A}_{L^\infty(\Omega)^3}$.
\end{proposition}

In this subsection we will apply  Proposition \ref{p1} in order to derive  Carleman estimates in negative order Sobolev space required for the construction of the CGO solutions.
For this purpose, we recall some preliminary tools.  Following \cite{Ki3,Ki5} (see also \cite{FKSU,ST}), for all $m\in\R$, we consider the space $H^m_\rho(\R^{3})$ defined by
\[H^m_\rho(\R^{3})=\{u\in\mathcal S'(\R^{3}):\ (|\xi|^2+\rho^2)^{m\over 2}\hat{u}\in L^2(\R^{3})\},\]
with the norm
\[\norm{u}_{H^m_\rho(\R^{3})}^2=\int_{\R^3}(|\xi|^2+\rho^2)^{m}|\hat{u}(\xi)|^2 d\xi .\]
In the above formula, for all tempered distribution $u\in \mathcal S'(\R^3)$,  $\hat{u}$ denotes the Fourier transform of $u$ which, for $u\in L^1(\R^3)$, corresponds to
$$\hat{u}(\xi):=\mathcal Fu(\xi):= (2\pi)^{-{3\over2}}\int_{\R^3}e^{-ix\cdot \xi}u(x)dx.$$
From now on, for $m\in\R$ and $\xi\in \R^3$,  we fix $$\left\langle \xi,\rho\right\rangle=(|\xi|^2+\rho^2)^{1\over2}$$
and $\left\langle D_x,\rho\right\rangle^m u$ given by
\[\left\langle D_x,\rho\right\rangle^m u=\mathcal F^{-1}(\left\langle \xi,\rho\right\rangle^m \mathcal Fu).\]
For $m\in\R$ we introduce also the class of symbols
\[S^m_\rho=\{c_\rho\in\mathcal C^\infty(\R^3\times\R^3):\ |\pd_x^\alpha\pd_\xi^\beta c_\rho(x,\xi)|\leq C_{\alpha,\beta}\left\langle \xi,\rho\right\rangle^{m-|\beta|},\  \alpha,\beta\in\mathbb N^3\}.\]
In light of \cite[Theorem 18.1.6]{Ho3}, for any $m\in\R$ and $c_\rho\in S^m_\rho$, we define $c_\rho(x,D_x)$, with  $D_x=-i\nabla $, by
\[c_\rho(x,D_x)y(x)=(2\pi)^{-{3\over 2}}\int_{\R^3}c_\rho(x,\xi)\hat{y}(\xi)e^{ix\cdot \xi} d\xi,\quad y\in\mathcal S(\R^3).\]
For all $m\in\R$, we set also $OpS^m_\rho:=\{c_\rho(x,D_x):\ c_\rho\in S^m_\rho\}$ and 
$$OpS^{-\infty}_\rho=\bigcap_{m<0} OpS^m_\rho.$$
We introduce also
$$P_{A,q,\pm}:=e^{\mp \rho x'\cdot\theta}(\Delta_{ A}+q) e^{\pm \rho x'\cdot\theta}.$$
According to \cite[Proposition 2.4]{Ki5}, there exists $C>0$, $\rho_*>1$, depending only on  $\Omega$ and $ \norm{q}_{L^\infty(\Omega))}+\norm{A}_{L^\infty(\Omega)^3}$, such that, for all $v\in \mathcal C^\infty_0(\Omega_1)$, the following estimate
$$\rho^{-1}\norm{v}_{H^{-1}_\rho(\R^3)}\leq C\norm{P_{A,q,\pm}v}_{H^{-1}_\rho(\R^3)},\quad \rho>\rho_*,$$
holds true. Using this estimate we can build CGO solutions lying in $H^1(\Omega)$. However, in order to improve the smoothness of our CGO solutions into functions lying in $H^2(\Omega)$, we will  consider the following  extension of \cite[Proposition 2.4]{Ki5}.

\begin{proposition}\label{p2} Let $A\in W^{1,\infty}(\Omega)^3$ and $q\in L^\infty(\Omega)$. Then, there exists $\rho_2>1$, depending only on $\Omega$ and $\norm{q}_{L^\infty(\Omega)}+\norm{A}_{W^{1,\infty}(\Omega)^3}$,   such that  for all $v\in \mathcal C^\infty_0(\Omega)$ we have 
\bel{p2a}\rho^{-1}\norm{v}_{L^2(\R^3)}\leq C\norm{P_{A,q,\pm}v}_{H^{-2}_\rho(\R^3)},\quad \rho>\rho_2,\ee
with $C>0$ depending on $\Omega$  and $\norm{q}_{L^\infty(\Omega)}+\norm{A}_{W^{1,\infty}(\Omega)^3}$.
\end{proposition}
\begin{proof} 
Without loss of generality, we will only show this result  for $P_{A,q,+}v$. We  fix
$$S_{ A,q,+,s}:=e^{-\phi_{+,s}}(\Delta_{ A}+q)e^{\phi_{+,s}}$$
and we split $S_{A,+,s}$ into three terms
\[S_{ A,q,+,s}=P_1+P_2+P_3,\]
where we recall that
\[P_1=\Delta+\rho^2-2s\rho (x'\cdot\theta)+s^2 (x'\cdot\theta)^2+s,\quad  P_2=2(\rho-s (x'\cdot\theta))\theta  \cdot\nabla-2s,\]
\[ P_3=2iA\cdot\nabla+2iA\cdot\nabla \phi_{+,s}+q-|A|^2+i\textrm{div}(A)=2iA\cdot\nabla+2(\rho-s(x'\cdot\theta)) iA'\cdot\theta +q-|A|^2+i\textrm{div}(A).\]
We choose $ \tilde{\omega}$  a bounded $\mathcal C^2$ open set of $\R^2$ such that $\overline{\omega}\subset\tilde{\omega}$ and we extend the function  $A$ and $q$ to $\R^3$ with   $q=0$ on $\R^3\setminus \Omega$ and $A\in W^{1,\infty}(\R^3)^3$ satisfying
$$\norm{A}_{W^{1,\infty}(\R^3)^3}\leq C\norm{A}_{W^{1,\infty}(\Omega)^3},$$
where $C>0$ depends only on $\Omega$. We consider also $\tilde{\Omega}=\tilde{\omega}\times\R$. We prove first the following the estimate
\bel{car}\rho^{-1}\norm{v}_{L^2(\R^3)}\leq C\norm{S_{ A,q,+,s}v}_{H^{-2}_\rho(\R^3)},\quad v\in\mathcal C^\infty_0(\Omega_1).\ee
For this purpose, we set $w\in H^4(\R^3)$ such that supp$(w)\subset\tilde{\Omega}$  and we consider
\[\left\langle D_x,\rho\right\rangle^{-2}(P_1+P_2)\left\langle D_x,\rho\right\rangle^2 w.\]
In all  this proof $C>0$ denotes a  constant depending on $\Omega$ and $\norm{A}_{W^{1,\infty}(\Omega)^3}+\norm{q}_{L^\infty(\Omega)}$.
According to the properties of composition of pseudoddifferential operators (e.g. \cite[Theorem 18.1.8]{Ho3}), we have
\bel{l2c}\left\langle D_x,\rho\right\rangle^{-2}(P_1+P_2)\left\langle D_x,\rho\right\rangle^2=P_1+P_2+R_\rho(x,D_x),\ee
where $R_\rho$ is defined by
\[R_\rho(x,\xi)=\nabla_\xi\left\langle \xi,\rho\right\rangle^{-2}\cdot D_x(p_1(x,\xi)+p_2(x,\xi))\left\langle \xi,\rho\right\rangle^2+\underset{\left\langle \xi,\rho\right\rangle\to+\infty}{ o}(1),\]
with
$$p_1(x,\xi)=-|\xi|^2+\rho^2-2s\rho (x'\cdot\theta)+s^2 (x'\cdot\theta)^2+s,\quad p_2(x,\xi)=2i[\rho-s(x'\cdot\theta)]\theta\cdot\xi'-2s,\quad \xi=(\xi',\xi_3)\in \R^2\times\R.$$
Therefore, one can check that
\bel{l2d} \norm{R_\rho(x,D_x)w}_{L^2( \R^3)}\leq Cs^2\norm{w}_{L^2( \R^3)}.\ee
Moreover, in view of \eqref{p1a} applied to $w$, with $\Omega$ replaced by  $\tilde{\Omega}$ and $A=0$, $q=0$,  we obtain
\[\norm{P_1w+P_2w}_{L^2(\R^3)}\geq C\left(s^{1/2}\rho^{-1}\norm{\Delta w}_{L^2(\R^3)}+s^{1/2}\norm{\nabla w}_{L^2(\R^3)}+s^{1/2}\rho\norm{ w}_{L^2(\R^3)}\right).\]
Combining this  with \eqref{l2c}-\eqref{l2d} and choosing ${\rho\over s^2}$ sufficiently large, we get
$$\begin{array}{l} \norm{(P_1+P_2)\left\langle D_x,\rho\right\rangle^2 w}_{H^{-2}_\rho(\R^3)}\\
=\norm{\left\langle D_x,\rho\right\rangle^{-2}(P_1+P_2)\left\langle D_x,\rho\right\rangle^2 w}_{L^2( \R^3)}\\ \geq  Cs^{1/2}\left(\rho^{-1}\norm{\Delta w}_{L^2(\R^3)}+\norm{\nabla w}_{L^2(\R^3)}+\rho\norm{ w}_{L^2(\R^3)}\right).\end{array}$$
Meanwhile, since $w\in H^2(\tilde{\Omega})\cap H^1_0(\tilde{\Omega})$, the elliptic regularity (e.g. \cite[Lemma 2.2]{CKS}) implies
$$\norm{w}_{H^2(\R^3)}=\norm{w}_{H^2(\tilde{\Omega})}\leq C(\norm{\Delta w}_{L^2(\tilde{\Omega})}+\norm{ w}_{L^2(\tilde{\Omega})}).$$
In view of the previous estimate, for $s$ sufficiently large, we have
 \bel{l2e}\norm{(P_1+P_2)\left\langle D_x,\rho\right\rangle^2 w}_{H^{-2}_\rho(\R^3)}\geq Cs^{\frac{1}{2}}\rho^{-1}\norm{w}_{H^2_\rho(\R^3)}.\ee
In addition, we find
\bel{p2c}\begin{aligned}&\norm{P_3\left\langle D_x,\rho\right\rangle^2 w}_{H^{-2}_\rho(\R^3)}\\
&\leq \norm{[2i(\rho-s(x'\cdot\theta))A\cdot\theta+(q-|A|^2)]\left\langle D_x,\rho\right\rangle^2 w}_{H^{-2}_\rho(\R^3)}+2\norm{A\cdot\nabla \left\langle D_x,\rho\right\rangle^2 w}_{H^{-2}_\rho(\R^3)}\\
&\ \ \ \ +\norm{i\textrm{div}(A)\left\langle D_x,\rho\right\rangle^2 w}_{H^{-2}_\rho(\R^3)}.\end{aligned}\ee
For the first term on the right hand side of this inequality, we have
\bel{p2d}\begin{aligned}\norm{[2i(\rho-s(x'\cdot\theta))A\cdot\theta+(q-|A|^2)]\left\langle D_x,\rho\right\rangle^2 w}_{H^{-2}_\rho(\R^3)}&\leq\rho^{-2}\norm{[2i(\rho-s(x'\cdot\theta))A\cdot\theta+(q-|A|^2)]\left\langle D_x,\rho\right\rangle^2 w}_{L^2(\R^3)}\\
\ &\leq\rho^{-1}C\norm{\left\langle D_x,\rho\right\rangle^2 w}_{L^2(\R^3)}\\
\ &\leq C\rho^{-1}\norm{ w}_{H^2_\rho(\R^3)},\end{aligned}\ee
with $C$ depending only on $\norm{A}_{W^{1,\infty}(\Omega_1)^3}+\norm{q}_{L^\infty(\Omega_1)}$. For the second term on the right hand side of \eqref{p2c}, we get
\bel{p2e}\begin{aligned}\norm{A\cdot\nabla \left\langle D,\rho\right\rangle^2 w}_{H^{-2}_\rho(\R^3)}&\leq \rho^{-1}\norm{A\cdot\nabla \left\langle D_x,\rho\right\rangle^2 w}_{H^{-1}(\R^3)}\\
\ &\leq \rho^{-1}\norm{A}_{W^{1,\infty}(\Omega_1)^3}\norm{\nabla \left\langle D,\rho\right\rangle^2 w}_{H^{-1}(\R^3)}\\
\ &\leq \rho^{-1}\norm{A}_{W^{1,\infty}(\Omega_1)^3}\norm{ w}_{H^{2}_\rho(\R^3)}.\end{aligned}\ee
Finally,  for the last term on the right hand side of \eqref{p2c},  we find
\bel{p2f}\begin{aligned}\norm{i\textrm{div}(A)\left\langle D_x,\rho\right\rangle^2 w}_{H^{-2}_\rho(\R^3)}&\leq \rho^{-2} \norm{i\textrm{div}(A)\left\langle D_x,\rho\right\rangle^2 w}_{L^2(\R^3)^3}\\
\ &\leq 3\rho^{-1}\norm{A}_{W^{1,\infty}(\Omega_1)^3}\norm{w}_{H^2_\rho(\R^3))}.\end{aligned}\ee
Combining the estimates \eqref{p2c}-\eqref{p2f}, we obtain
$$\norm{P_3\left\langle D_x,\rho\right\rangle^2 w}_{H^{-2}_\rho(\R^3)}\leq C\rho^{-1}\norm{w}_{H^2_\rho(\R^3)}$$
and applying \eqref{l2e} for $s>1$ sufficiently large, we have
\bel{p2g}\norm{R_{ A,q,+,s}\left\langle D_x,\rho\right\rangle^2 w}_{H^{-2}_\rho(\R^3)}^2\geq Cs^{\frac{1}{2}}\rho^{-1}\norm{w}_{H^2_\rho(\R^3)}.\ee
Now let us fix $\omega_j$, $j=1,2$ two open subsets of $\tilde{\omega}$  such that $\overline{\omega}\subset \omega_1$, $\overline{\omega_1}\subset \omega_2$ and $\overline{\omega_2}\subset \tilde{\omega}$.
We consider $\psi_0\in\mathcal C^\infty_0(\tilde{\omega})$ satisfying $\psi_0=1$ on $\overline{\omega_2}$, $v\in\mathcal C^\infty_0(\Omega)$, $w(x',x_3)=\psi_0(x') \left\langle D_x,\rho\right\rangle^{-2} v(x',x_3)$ and  $\psi_1\in\mathcal C^\infty_0(\omega_1)$ satisfying $\psi_1=1$ on $\omega$. Then,  we have $$(1-\psi_0 )\left\langle D_x,\rho\right\rangle^{-1} v=(1-\psi_0 )\left\langle D_x,\rho\right\rangle^{-2}\psi_1 v,$$
where $\psi_1v:=(x',x_3)\mapsto \psi_1(x')v(x',x_3)$. In view of to \cite[Theorem 18.1.8]{Ho3}, using the fact that $1-\psi_0$ is vanishing in a neighborhood of supp$(\psi_1)$, we have $(1-\psi_0) \left\langle D_x,\rho\right\rangle^{-1}\psi_1\in OpS^{-\infty}_\rho$ and we get
\[ \begin{aligned}\rho^{-1}\norm{v}_{L^2(\R^3)}&=\rho^{-1}\norm{\left\langle D_x,\rho\right\rangle^{-2} v}_{H^2_\rho(\R^3)}\\
\ &\leq \rho^{-1}\norm{w}_{H^2_\rho(\R^3)}+\rho^{-1}\norm{(1-\psi_0)\left\langle D_x,\rho\right\rangle^{-2}\psi_1 v}_{H^2_\rho(\R^3)}\\
\ &\leq \rho^{-1}\norm{w}_{H^2_\rho(\R^3)}+{C\norm{v}_{L^2(\R^3)}\over\rho^2} .\end{aligned}\]
In the same way, we obtain
$$\begin{aligned}\norm{P_{A,-,s} v}_{H^{-2}_\rho(\R^3)}&\geq \norm{P_{A,-,s}\left\langle D_x,\rho\right\rangle^2 w}_{H^{-2}_\rho(\R^3)}-\norm{P_{A,-,s}\left\langle D_x,\rho\right\rangle^2 (1-\psi_0 )\left\langle D_x,\rho\right\rangle^{-2}\psi_1 v}_{H^{-1}_\rho(\R^3)}\\
\ &\geq \norm{P_{A,-,s}\left\langle D_x,\rho\right\rangle^2 w}_{H^{-2}_\rho(\R^3)}-C\norm{ (1-\psi_0 )\left\langle D_x,\rho\right\rangle^{-2}\psi_1 v}_{H^{2}_\rho(\R^3)}\\
\ &\geq \norm{P_{A,-,s}\left\langle D_x,\rho\right\rangle^2 w}_{H^{-2}_\rho(\R^3)}-{C\norm{v}_{L^2(\R^{1+n})}\over\rho^2}.\end{aligned}$$
Combining these estimates with \eqref{p2g}, we deduce that \eqref{car} holds true for a sufficiently large value of  $\rho$. Then, fixing $s$, we deduce \eqref{p2a}. \end{proof}

\subsection{Remainder term of the CGO solutions}
In this subsection we will construct the remainder term $w_{j,\rho}$, $j=1,2$, appearing in \eqref{CGO1}-\eqref{CGO2} and satisfying the decay property \eqref{RCGO}. For this purpose, we will combine the Carleman estimate \eqref{p2a} with the properties of the expressions $b_j$, $j=1,2$, in order to complete the construction of these solutions.  In this subsection, we assume that $\rho>\rho_2$ with $\rho_2$ the constant introduced in Proposition \ref{p2} and we fix $A_j\in (W^{2,\infty} (\Omega ))^3$, $j=1,2$, satisfying \eqref{t1a}-\eqref{t1b}. The proof for the existence of the remainder term $w_{1,\rho}$ and $w_{2,\rho}$ being similar, we will only show the existence of $w_{1,\rho}$. Let us first remark that $w_{1,\rho}$ should be a solution of the equation
\bel{t4a} P_{A_1,q_1,+}w=e^{-\rho \theta\cdot x'}(-\Delta_{\tilde{A}_1}+q_1)e^{\rho \theta\cdot x'}w=e^{i\rho \eta\cdot x}F_{1,\rho}(x),\quad x\in\Omega,\ee
with $F_{1,\rho}$ defined, for all $x=(x',x_3)\in B'_{r_0+1}\times\R$ (we recall that $B_r'=\{x'\in\R^{2}:\ |x'|<r\}$), by
\bel{t4b}\begin{aligned}F_{1,\rho}(x)&=-e^{-\rho \theta\cdot x'-i\rho \eta\cdot x}(\Delta_{\tilde{A}_1}+q_1)\left[e^{\rho \theta\cdot x'+i\rho \eta\cdot x}\psi\left(\rho^{-\frac{1}{4}}x_3\right)b_1e^{-i\xi\cdot x}\right]\\
&=-\left((|\xi|^2+q_1)\psi\left(\rho^{-\frac{1}{4}}x_3\right)-2i\eta_3\rho^{\frac{3}{4}}\psi'\left(\rho^{-\frac{1}{4}}x_3\right)-2i\xi_3\rho^{-\frac{1}{4}}\psi'\left(\rho^{-\frac{1}{4}}x_3\right)\right)b_1e^{-i\xi\cdot x}\\
&\ \ \  -\left[-\rho^{-\frac{1}{2}}\psi''\left(\rho^{-\frac{1}{4}}x_3\right)b_1- 2\pd_{x_3}b_1\rho^{-\frac{1}{4}}\psi'\left(\rho^{-\frac{1}{4}}x_3\right)+[i2\xi\cdot\nabla b_1-\Delta_{\tilde{A}_1} b_1]\psi\left(\rho^{-\frac{1}{4}}x_3\right)\right]e^{-i\xi\cdot x}.\end{aligned}\ee
Here we have used \eqref{tt} and we consider $q_1$ as a function extended by zero to $\R^3$.
 We fix $\phi\in\mathcal C^\infty_0(B'_{r_0+1};[0,1])$ satisfying $\phi=1$ on $B'_{r_0+\frac{1}{2}}$, and we define $$G_\rho(x',x_3):=\phi(x')F_{1,\rho}(x',x_3),\quad x'\in\R^2,\ x_3\in\R.$$
It is clear that $G_\rho\in L^2(\R^3)$ and in view of \eqref{l1a} and the fact that
$$\norm{\psi\left(\rho^{-\frac{1}{4}}x_3\right)}_{L^2(B'_{r_0+1}\times\R)}+\norm{\psi'\left(\rho^{-\frac{1}{4}}x_3\right)}_{L^2(B'_{r_0+1}\times\R)}+\norm{\psi''\left(\rho^{-\frac{1}{4}}x_3\right)}_{L^2(B'_{r_0+1}\times\R)}\leq C\rho^{\frac{1}{8}},$$
we deduce that
\bel{t4h}\norm{G_\rho}_{L^2(\R^3)}\leq C(|\xi|^2+1)\left(1+\frac{|\xi'|}{|\xi_3|}\right)\rho^{\frac{7}{8}} ,\ee
with $C>0$ depending on $\Omega$ and $M$.
From now on we denote by $C>0$ a constant depending only on $\Omega$ and $M$ that may change from line to line. Applying \eqref{p2a}, we will complete the construction of the remainder term $w_{1,\rho}$ by using a classical duality argument. More precisely, applying \eqref{p2a}, we consider the linear form 
$ T_\rho$ defined on $\mathcal Q:=\{P_{A_1,\overline{q_1},-}w: w\in\mathcal C^\infty_0(\Omega)\}$ by
$$T_\rho(P_{A_1,q_1,-}v):=\overline{\left\langle G_{\rho}, e^{-i\rho \eta\cdot x}v\right\rangle_{H^{-2}_\rho(\R^3), H^2_\rho(\R^3)}},\quad v\in\mathcal C^\infty_0(\Omega).$$
Here and from now on we define the duality bracket $\left\langle \cdot,\cdot\right\rangle_{H^{-2}_\rho(\R^3), H^2_\rho(\R^3)}$ in the complex sense, which means that 
$$\left\langle v,w\right\rangle_{H^{-2}_\rho(\R^3), H^2_\rho(\R^3)}=\left\langle v,w\right\rangle_{L^2(\R^3)}=\int_{\R^3}v\overline{w}dx,\quad v\in L^2(\R^3),\ w\in H^2(\R^3).$$
Applying again \eqref{p2a}, for all $v\in\mathcal C^\infty_0(\Omega)$, we obtain
$$\begin{aligned}|T_\rho(P_{A_1,q_1,-}v)|&\leq \norm{G_{\rho}}_{L^2(\R^3)}\norm{e^{-i\rho \eta\cdot x}v}_{L^2(\R^3)}\\
\ &\leq C\rho\norm{G_{\rho}}_{L^2(\R^3)}\rho^{-1}\norm{v}_{L^2(\R^3)}\\
\ &\leq C\rho\norm{G_{\rho}}_{L^2(\R^3)}\norm{P_{A_1,\overline{q_1},-}v}_{H^{-2}_\rho(\R^3)},\end{aligned}$$
with $C>0$ depending on $\Omega$ and $\norm{q_1}_{L^\infty(\Omega)}+\norm{A_1}_{W^{1,\infty}(\Omega)^3}$. Thus, applying the Hahn-Banach theorem, we deduce that $T_\rho$ admits an extension as a continuous linear form on ${H^{-2}_\rho(\R^3)}$ whose norm will be upper bounded by $C\rho\norm{G_{\rho}}_{L^2(\R^3)}$. Therefore, there exists $w_{1,\rho}\in H^2_\rho(\R^3)$ such that
\bel{t4d}\left\langle P_{A_1,\overline{q_1},-}v, w_{1,\rho}\right\rangle_{H^{-2}_\rho(\R^3), H^2_\rho(\R^3)}=T_\rho(P_{A_1,\overline{q_1},-}v)=\overline{\left\langle G_{\rho}, e^{-i\rho \eta\cdot x}v\right\rangle_{H^{-1}_\rho(\R^3), H^1_\rho(\R^3)}},\quad v\in\mathcal C^\infty_0(\Omega),\ee
\bel{t4e}\norm{w_{1,\rho}}_{H^2_\rho(\R^3)}\leq C\rho\norm{G_{\rho}}_{L^2(\R^3)}.\ee
From \eqref{t4d} and the fact that, for all $x\in\Omega$, $G_\rho(x)=F_{1,\rho}(x)$, we obtain
$$\left\langle  P_{A_1,q_1,+}w_{1,\rho},v\right\rangle_{D'(\Omega), \mathcal C^\infty_0(\Omega)}=\left\langle e^{i\rho \eta\cdot x}F_{1,\rho}, v \right\rangle_{D'(\Omega), \mathcal C^\infty_0(\Omega)}.$$
It follows that $w_{1,\rho}$ solves 
$P_{A_1,q_1,+}w_{1,\rho}=e^{i\rho \eta\cdot x}F_{1,\rho}$ in $\Omega$ and $u_1$ given by \eqref{CGO1} is a solution of $\Delta_{A_1}u+q_1u=0$ in $\Omega$ lying in $H^2(\Omega)$. In addition, from  \eqref{t4e}, we deduce that
$$\rho^{-1}\norm{w_{1,\rho}}_{H^2(\Omega)}+\norm{w_{1,\rho}}_{H^1(\Omega)}+\rho\norm{w_{1,\rho}}_{L^2(\Omega)}\leq C(|\xi|^2+1)\left(1+\frac{|\xi'|}{|\xi_3|}\right)\rho^{\frac{7}{8}}$$
which implies the decay property \eqref{RCGO}. This completes the proof of Theorem \ref{t4}.

\section{Stability results on the whole boundary}

This section is devoted to the proof of our results with full boundary measurements stated in Theorem \ref{t1} and \ref{t2}.

\subsection{Recovery of the magnetic field}
In this subsection we will prove Theorem \ref{t1}. In all this proof $C$ and $c$ will be two positive constants depending only on $\Omega$ and $M$ that may change from line to line.

For $j=1,2$, we fix $u_j\in H^{2}(\Omega)$  a solution of  $\Delta_{A_1} u_1+q_1u_1=0$, $\Delta_{A_2} u_2+\overline{q_2}u_2=0$ in $\Omega$ of the form \eqref{CGO1}-\eqref{CGO2} with $\rho>\rho_2$ and with $w_{j,\rho}$ satisfying \eqref{RCGO}. These solutions satisfy the following property
\begin{lemma}\label{lem3}
There exists $C>0$ such that the following estimates  
 $$\Vert u_1 \Vert_{H^2(\Omega)} \leqslant C  e^{(D+1)\rho} \left(1+\frac{|\xi'|}{|\xi_3|}\right) \Big( 1+|\xi|^2 \Big),$$ $$ \Vert u_2 \Vert_{H^2(\Omega)} \leqslant C  e^{(D+1)\rho} \left(1+\frac{|\xi'|}{|\xi_3|}\right) \Big( 1+|\xi|^2 \Big),$$
hold true for any solutions $u_1$ and $u_2$ given by $(\ref{CGO1})$ and $(\ref{CGO2})$.\\
Here $D:= \underset{x' \in \overline{\omega}}{sup} \vert x'\vert$.
\end{lemma}
\begin{proof}
Using the expression of $u_1$, we can easily deduce that 
$$\Vert u_1 \Vert_{L^2(\Omega)} \leqslant \Vert e^{\rho\theta\cdot x'} \Vert_{L^\infty(\Omega)} \Vert \psi\big( \rho^{-\frac{1}{4}}x_3 \big) b_1  e^{i\rho\eta\cdot x-i\xi\cdot x} + w_{1,\rho} (x' , x_3) \Vert_{L^2(\Omega)}.$$
Setting $D:= \underset{x' \in \overline{\omega}}{sup} \vert x'\vert$ and using \eqref{RCGO} and \eqref{cond31}, we get
\begin{align*}
\Vert u_1 \Vert_{L^2(\Omega)} & \leqslant C e^{D\rho} \Big( \rho^{\frac{1}{8}} \Vert A_1 \Vert_{W^{2, \infty}(\Omega)^3} \left(1+\frac{|\xi'|}{|\xi_3|}\right) +\rho^{-\frac{1}{8}} \Big( 1+|\xi|^2 \Big) \left(1+\frac{|\xi'|}{|\xi_3|}\right) \Big)   \\&  \leqslant C e^{(D+1)\rho} \left(1+\frac{|\xi'|}{|\xi_3|}\right) \Big( 1+|\xi|^2 \Big).
\end{align*}
By simple computations of $\nabla u_1$ and $\partial_{x_i}\partial_{x_j}$, for $i,j=1,2,3$ and by the same arguments used previously, we obtain
\begin{align*}
\Vert \nabla u_1 \Vert_{L^2(\Omega)} & \leqslant  e^{D\rho} \Big( (\rho + \vert \xi \vert) \Vert \psi \big( \rho^{-\frac{1}{4}}x_3 \big) b_1 \Vert_{L^2(\Omega)} +  \Vert \psi \big( \rho^{-\frac{1}{4}}x_3 \big) \nabla b_1 \Vert_{L^2(\Omega)} + \rho^{-\frac{1}{4}}  \Vert \psi'\big( \rho^{-\frac{1}{4}}x_3 b_1 \big) \Vert_{L^2(\Omega)} \\ & \quad + \rho \Vert  w_{1, \rho} \Vert_{L^2(\Omega)} + \Vert \nabla w_{1, \rho} \Vert_{L^2(\Omega)}\Big) \\
& \leqslant C e^{(D+1)\rho} \left(1+\frac{|\xi'|}{|\xi_3|}\right) \Big( 1+|\xi|^2 \Big)
\end{align*}
and 
\begin{align*}
\Vert \partial_{x_i}\partial_{x_j} u_1 \Vert_{L^2(\Omega)} & \leqslant e^{D\rho} \Big(  (\rho + \vert \xi \vert)^2 \Vert \psi\big( \rho^{-\frac{1}{4}}x_3 \big) b_1 \Vert_{L^2(\Omega)} + \rho^{-\frac{1}{4}}(\rho + \vert \xi \vert) \Vert \psi'\big( \rho^{-\frac{1}{4}}x_3 \big) b_1 \Vert_{L^2(\Omega)} \\ \quad & + (\rho + \vert \xi \vert) \Vert\psi\big( \rho^{-\frac{1}{4}}x_3 \big) \nabla b_1 \Vert_{L^2(\Omega)} +\rho^{-\frac{1}{4}} \Vert \psi'\big( \rho^{-\frac{1}{4}}x_3 \big) \nabla b_1 \Vert_{L^2(\Omega)} + \Vert \psi\big( \rho^{-\frac{1}{4}}x_3 \big) \partial_{x_i}\partial_{x_j} b_1 \Vert_{L^2(\Omega)} \\
& \quad + \rho^{-\frac{1}{2}} \Vert \psi''\big( \rho^{-\frac{1}{4}}x_3 \big) b_1 \Vert_{L^2(\Omega)} + \rho^2 \Vert  w_{1, \rho} \Vert_{L^2(\Omega)} + \rho \Vert \nabla w_{1, \rho} \Vert_{L^2(\Omega)} +  \Vert \partial_{x_i}\partial_{x_j} w_{1, \rho} \Vert_{L^2(\Omega)} \Big) \\
& \leqslant C e^{(D+1)\rho} \left(1+\frac{|\xi'|}{|\xi_3|}\right) \Big( 1+|\xi|^2 \Big)
\end{align*}
In the same way, we get
$$\Vert u_2 \Vert_{L^2(\Omega)} \leqslant C  e^{(c+1)\rho} \left(1+\frac{|\xi'|}{|\xi_3|}\right) \Big( 1+|\xi|^2 \Big) $$  and $$ \qquad \Vert  u_2 \Vert_{H^2(\Omega)} \leqslant C  e^{(c+1)\rho} \left(1+\frac{|\xi'|}{|\xi_3|}\right) \Big(  1+|\xi|^2 \Big) . $$
This completes the proof.
\end{proof}

Fixing $q=q_1-q_2$ extended by zero to an element of $ L^\infty(\R^3)$ and applying a classical integration by parts argument, we deduce the following identity
\bel{tt1aa}\begin{aligned}&\left\langle (\Lambda_{A_1,q_1}- \Lambda_{A_2,q_2})u_1, u_2\right\rangle_{L^2(\partial\Omega )}
\ &=i\int_{\R^3}(A\cdot\nabla u_1)\overline{u_2}dx -i\int_{\R^3}u_1(\overline{A\cdot\nabla u_2})dx+\int_{\R^3}\tilde{q}u_1\overline{u_2}dx,\end{aligned}\ee
where $\tilde{q}=|A_2|^2-|A_1|^2+q$. 

By simple computations, we get
\begin{align*}
\nabla u_1 \overline{u}_2 &- u_1 \nabla \overline{u}_2 \\ &= 2 \rho (\overset{\sim}{\theta} + i \eta ) \psi\left(\rho^{-\frac{1}{4}}x_3\right)^2 b_1 \overline{b}_2 e^{-ix\cdot\xi} + \rho (2\overset{\sim}{\theta} + i \eta ) \psi\left(\rho^{-\frac{1}{4}}x_3\right) \left( b_1 e^{i \rho x \cdot \eta - i \xi \cdot x } w_{2, \rho} + \overline{b}_2 e^{-i \rho x \cdot \eta  } w_{1, \rho}\right)
\\ & - i \xi  \psi\left(\rho^{-\frac{1}{4}}x_3\right) b_1 \left(  \psi\left(\rho^{-\frac{1}{4}}x_3\right) \overline{b}_2 e^{-ix\cdot\xi} + e^{i \rho x \cdot \eta - i \xi \cdot x } w_{2, \rho} \right) + \psi\left(\rho^{-\frac{1}{4}}x_3\right)^2 e^{-ix\cdot\xi} \left( \overline{b}_2 \nabla b_1 - b_1 \nabla \overline{b}_2  \right) \\
& + \psi\left(\rho^{-\frac{1}{4}}x_3\right) \left[  e^{i \rho x \cdot \eta - i \xi \cdot x } \left( \nabla b_1 w_{2, \rho} - b_1 \nabla w_{2, \rho} \right) + e^{-i \rho x \cdot \eta  } \left( \overline{b}_2 \nabla w_{1, \rho} - \nabla\overline{b}_2 w_{1, \rho}\right) \right] \\
& + \rho^{-\frac{1}{4}} \partial_{x_3} \psi\left(\rho^{-\frac{1}{4}}x_3\right) \left( b_1 e^{i \rho x \cdot \eta - i \xi \cdot x }w_{2, \rho} - b_2 e^{-i \rho x \cdot \eta  }w_{1, \rho}  \right) + 2 \rho \overset{\sim}{\theta} w_{1, \rho} w_{2, \rho} + \nabla w_{1, \rho}  w_{2, \rho} - w_{1, \rho} \nabla w_{2, \rho} .
\end{align*}

According to \eqref{RCGO},  \eqref{l1a}, Lemma \ref{lem3} and the fact that $A\in L^1(\R^3)$, multiplying  this expression by $-i\rho^{-1}2^{-1}$, we find
\bel{t1aa}\begin{aligned}&\abs{\int_{\R^3}(A\cdot(\tilde{\theta}+i\eta))\psi\left(\rho^{-\frac{1}{4}}x_3\right)^2\exp\left(\Phi_1+\overline{\Phi_2}\right)e^{-ix\cdot\xi}dx}\\
&\leq C\rho^{-1}\abs{\int_{\Omega}\left(\psi\left(\rho^{-\frac{1}{4}}x_3\right)b_1e^{i\rho x\cdot\eta-i\xi\cdot x}+w_{1,\rho}(x',x_3)\right)\left(\psi\left(\rho^{-\frac{1}{4}}x_3\right)b_2 e^{i\rho x\cdot\eta}+w_{2,\rho}(x',x_3)\right)dx}\\
&\ \ \  +C\left(1+\frac{|\xi'|}{|\xi_3|}\right)^2(1+|\xi|^2)\left[\rho^{-\frac{1}{4}}+\norm{\Lambda_{A_1,q_1}- \Lambda_{A_2,q_2}}_{\mathcal B(H^{\frac{3}{2}}(\partial\Omega),L^2(\partial\Omega))}e^{c\rho}\right]\\
&\leq C\left(1+\frac{|\xi'|}{|\xi_3|}\right)^2(1+|\xi|^2)\left[\rho^{-\frac{1}{4}}+\norm{\Lambda_{A_1,q_1}- \Lambda_{A_2,q_2}}_{\mathcal B(H^{\frac{3}{2}}(\partial\Omega),L^2(\partial\Omega))}e^{c\rho}\right],\end{aligned}\ee
where $$c=2(\sup_{x'\in\omega}|x'|+1).$$ 
On the other hand, using \eqref{t1b}, \eqref{l1b} and the fact that $\psi=1$ on $[-1,1]$ and $0\leq\psi\leq1$, we get
$$\begin{aligned}&\abs{\int_{\R^3}(A\cdot(\tilde{\theta}+i\eta))\psi\left(\rho^{-\frac{1}{4}}x_3\right)^2\exp\left(\Phi_1+\overline{\Phi_2}\right)e^{-ix\cdot\xi}dx}\\
&\geq \abs{\int_{\R^3}(A\cdot(\tilde{\theta}+i\eta))\exp\left(\Phi_1+\overline{\Phi_2}\right)e^{-ix\cdot\xi}dx}-C\int_{\R^3}\abs{A}(1-\psi\left(\rho^{-\frac{1}{4}}x_3\right)^2)dx\\
&\geq \abs{\int_{\R^3}(A\cdot(\tilde{\theta}+i\eta))\exp\left(\Phi_1+\overline{\Phi_2}\right)e^{-ix\cdot\xi}dx}-C\int_{\R^2}\int_{|x_3|\geq\rho^{\frac{1}{4}}}\left\langle x_3\right\rangle^{-s}(\left\langle x_3\right\rangle^s\abs{A})dx_3dx'\\
&\geq \abs{\int_{\R^3}(A\cdot(\tilde{\theta}+i\eta))\exp\left(\Phi_1+\overline{\Phi_2}\right)e^{-ix\cdot\xi}dx}-C\rho^{-\frac{s}{4}}\int_\Omega\left\langle x_3\right\rangle^s(\abs{A_1}+\abs{A_2})dx\\
&\geq \abs{\int_{\R^3}(A\cdot(\tilde{\theta}+i\eta))\exp\left(\Phi_1+\overline{\Phi_2}\right)e^{-ix\cdot\xi}dx}-2CM\rho^{-\frac{s}{4}}.\end{aligned}$$
Combining this with \eqref{t1aa}, we obtain 
\bel{t1bb}\begin{aligned}&\abs{\int_{\R^3}(A\cdot(\tilde{\theta}+i\eta)) \exp\left(\Phi_1+\overline{\Phi_2}\right)e^{-ix\cdot\xi}dx}\\ 
& \leq \abs{\int_{\R^3}(A\cdot(\tilde{\theta}+i\eta))\psi \left(\rho^{-\frac{1}{4}}x_3\right)^2\exp\left(\Phi_1+\overline{\Phi_2}\right)e^{-ix\cdot\xi}dx}+2CM\rho^{-\frac{s}{4}}    \\
&\leq C\left(1+\frac{|\xi'|}{|\xi_3|}\right)^2(1+|\xi|^2)\left[\rho^{-\frac{s}{4}}+\norm{\Lambda_{A_1,q_1}- \Lambda_{A_2,q_2}}_{\mathcal B(H^{\frac{3}{2}}(\partial\Omega),L^2(\partial\Omega))}e^{c\rho}\right]+2CM\rho^{-\frac{s}{4}}.\end{aligned}\ee 
Now let us observe that
$$\Phi:=\Phi_1+\overline{\Phi_2}=\frac{-i}{2\pi} \int_{\R^2} \frac{(\tilde{\theta}+i\eta)\cdot A(x-s_1\tilde{\theta}-s_2\eta)}{s_1+is_2}ds_1ds_2.$$
Therefore, we have 
$$\abs{\int_{\R^3}(A\cdot(\tilde{\theta}+i\eta))e^\Phi e^{-ix\cdot\xi}dx}\leq C\left(1+\frac{|\xi'|}{|\xi_3|}\right)^2(1+|\xi|^2)\left[\rho^{-\frac{s}{4}}+\norm{\Lambda_{A_1,q_1}- \Lambda_{A_2,q_2}}_{\mathcal B(H^{\frac{3}{2}}(\partial\Omega),L^2(\partial\Omega))}e^{c\rho}\right].$$
Applying \cite[Lemma 4.1]{Ki5}, we deduce from the above estimate that
\bel{t1cc}\abs{(\tilde{\theta}+i\eta)\cdot\mathcal F(A)(\xi)}\leq C\left(1+\frac{|\xi'|}{|\xi_3|}\right)^2(1+|\xi|^2)\left[\rho^{-\frac{s}{4}}+\norm{\Lambda_{A_1,q_1}- \Lambda_{A_2,q_2}}_{\mathcal B(H^{\frac{3}{2}}(\partial\Omega),L^2(\partial\Omega))}e^{c\rho}\right].\ee
In the same way, replacing $\eta$ by $-\eta$ in the construction of the CGO $u_j$, $j=1,2$, we obtain
$$\abs{(\tilde{\theta}-i\eta)\cdot\mathcal F(A)(\xi)}\leq C\left(1+\frac{|\xi'|}{|\xi_3|}\right)^2(1+|\xi|^2)\left[\rho^{-\frac{s}{4}}+\norm{\Lambda_{A_1,q_1}- \Lambda_{A_2,q_2}}_{\mathcal B(H^{\frac{3}{2}}(\partial\Omega),L^2(\partial\Omega))}e^{c\rho}\right].$$
Combining these two estimates with the fact that $(\tilde{\theta},\eta)$ is an orthonormal basis of $\xi^\bot=\{y\in\R^3:\ y\cdot\xi=0\}$, we find 

\bel{t1dd}\abs{\zeta\cdot\mathcal F(A)(\xi)}\leq C|\zeta|\left(1+\frac{|\xi'|}{|\xi_3|}\right)^2(1+|\xi|^2)\left[\rho^{-\frac{s}{4}}+\norm{\Lambda_{A_1,q_1}- \Lambda_{A_2,q_2}}_{\mathcal B(H^{\frac{3}{2}}(\partial\Omega),L^2(\partial\Omega))}e^{c\rho}\right],\quad \zeta\in\xi^\bot.\ee
Moreover, for $1\leq j<k\leq3$, fixing $\zeta=\xi_ke_j-\xi_j e_k,$
with $$e_j=(0,\ldots,0,\underbrace{1}_{\textrm{position\ } j},0,\ldots0),\quad e_k=(0,\ldots,0,\underbrace{1}_{\textrm{position } k},0,\ldots0),$$ \eqref{t1dd} implies
\bel{t1f}\abs{\xi_k \mathcal F(a_j)(\xi)-\xi_j \mathcal F(a_k)(\xi)}\leq C\left(1+\frac{|\xi'|}{|\xi_3|}\right)^2(1+|\xi|^3)\left[\rho^{-\frac{s}{4}}+\norm{\Lambda_{A_1,q_1}- \Lambda_{A_2,q_2}}_{\mathcal B(H^{\frac{3}{2}}(\partial\Omega),L^2(\partial\Omega))}e^{c\rho}\right],\ee
where $A=(a_1,a_2,a_3)$. Recall that so far, we have proved \eqref{t1f} for any $\xi=(\xi',\xi)\in\R^2\times\R$ with $\xi'\neq0$ and $\xi_3\neq0$. Then, we deduce from \eqref{t1f} that
$$\abs{\mathcal F(\pd_{x_k}a_j-\pd_{x_j}a_k)(\xi)}\leq C\left(1+\frac{|\xi'|}{|\xi_3|}\right)^2(1+|\xi|^3)\left[\rho^{-\frac{s}{4}}+\norm{\Lambda_{A_1,q_1}- \Lambda_{A_2,q_2}}_{\mathcal B(H^{\frac{3}{2}}(\partial\Omega),L^2(\partial\Omega))}e^{c\rho}\right] .$$
From now on we fix $R>1$, $\gamma:= \norm{\Lambda_{A_1,q_1}- \Lambda_{A_2,q_2}}_{\mathcal B(H^{\frac{3}{2}}(\partial\Omega),L^2(\partial\Omega))}$ and we consider the set
$$D_R=\{\xi=(\xi_1,\xi_2,\xi_3)\in \R^3:\ |\xi|\leq R,\ |\xi_3|\geq R^{-4}\}.$$
We obtain the estimate
$$\abs{\mathcal F(\pd_{x_k}a_j-\pd_{x_j}a_k)(\xi)}\leq C(R^{13}\rho^{-\frac{s}{4}}+R^{13}\gamma e^{c\rho}),\quad \xi\in D_R.$$
It follows that
\bel{t1ee}\int_{D_R}\abs{\mathcal F(\pd_{x_k}a_j-\pd_{x_j}a_k)(\xi)}^2d\xi \leq C(R^{29}\rho^{-\frac{s}{2}}+R^{29}\gamma^2 e^{2c\rho}).\ee
On the other hand, using the fact that $A_1-A_2\in W^{1,1}(\Omega)^3$ satisfies \eqref{t1a}, we obtain
$$\norm{\mathcal F(\pd_{x_k}a_j-\pd_{x_j}a_k)}_{L^\infty(\R^3)}\leq 2\norm{A}_{W^{1,1}(\R^3)^3}\leq 2\norm{A_1-A_2}_{W^{1,1}(\Omega)^3}\leq 2M.$$
Therefore, we have 
\bel{t1ff}\int_{B_R\setminus D_R}\abs{\mathcal F(\pd_{x_k}a_j-\pd_{x_j}a_k)(\xi)}^2d\xi\leq 4M^2\int_{-R^{-4}}^{R^{-4}}\int_{B'_R}d\xi'd\xi_3\leq CR^{-2}.\ee
In the same way, using the fact that $A_j\in H^2(\Omega)^3$ and applying \eqref{t1a}, we deduce that $A\in H^2(\R^3)^3$ and \eqref{t1b} implies that 
$$\norm{A}_{H^2(\R^3)^3}\leq 2M.$$
Applying this estimate we deduce that 
\bel{t1gg}\begin{aligned}\int_{\R^3\setminus B_R}\abs{\mathcal F(\pd_{x_k}a_j-\pd_{x_j}a_k)(\xi)}^2d\xi&\leq CR^{-2}\int_{\R^3}(1+|\xi|^2)\abs{\mathcal F(\pd_{x_k}a_j-\pd_{x_j}a_k)(\xi)}^2d\xi\\
&\leq 2CMR^{-2}.\end{aligned}\ee
Combining \eqref{t1ee}-\eqref{t1gg}, we get
$$\int_{\R^3}\abs{\mathcal F(\pd_{x_k}a_j-\pd_{x_j}a_k)(\xi)}^2d\xi\leq C(R^{-2}+R^{29}\rho^{-\frac{s}{2}}+R^{29}\gamma^2 e^{2c\rho})$$
and by Plancherel formula, it follows
$$\norm{dA}_{L^2(\Omega)}\leq C(R^{-1}+R^{29/2}\rho^{-\frac{s}{4}}+R^{29/2}\gamma e^{c\rho}).$$
Choosing $R=\rho^{\frac{s}{62}}$, we get
\begin{align}\label{32}
\norm{dA}_{L^2(\Omega)}&\leq C(\rho^{-\frac{s}{62}}+(\rho^{-\frac{s}{62}})^{29/2}\gamma e^{c\rho}) \nonumber \\ &\leq C(\rho^{-\frac{s}{62}}+\gamma e^{(c+1)\rho}).
\end{align}
Now, let us recall a classical result already stated in \cite{So}.
\begin{lemma}\label{min}
Let $a \in (0,1]$ and $b>0$. Then, there exists $C>0$ depending only on $b$, such that $$ \underset{\rho >1}{inf } \quad \rho^{-\frac{s}{62}}+a e^{b\rho} \leqslant C (\log(3+a^{-1}))^{-\frac{s}{62}} .$$
\end{lemma}
Combining \eqref{32} with Lemma \ref{min}, for $\gamma \leqslant 1$, we obtain 
\begin{equation}\label{3333}
\Vert dA \Vert_{L^2(\R^3)} \leqslant C (\log (3 +\gamma^{-1}))^{-\frac{s}{62}} .
\end{equation}
In the same way, for $\gamma \geqslant 1$, we have 
\begin{align*}
\Vert dA \Vert_{L^2(\R^3)} & \leqslant 2M \log (4)^{\frac{s}{62}}(\log (3 +\gamma^{-1}))^{-\frac{s}{62}} \\
& \leqslant C (\log (3 +\gamma^{-1}))^{-\frac{s}{62}}.
\end{align*}
Combining this estimate with \eqref{3333}, we deduce that \eqref{t1d} holds true for $\gamma >0$.
\\
For $\gamma = 0 $, $\eqref{32} $ implies that $\Vert dA \Vert_{L^2(\R^3)} \leqslant C \rho^{-\frac{s}{62}}$. Since $\rho >1$ is arbitrary, we can send $\rho $ to $+\infty$ and deduce \eqref{t1d} for $\gamma =0$.
This completes the proof.
\subsection{Recovery of the electric potential}
In this subsection we assume that \eqref{t1a}-\eqref{t2a} hold true and we will show \eqref{t2b}. In all this proof $C$ and $c$ will be two positive constants depending only on $\Omega$ and $M$ that may change from line to line. For this purpose, we start by proving the following estimate
\bel{t2d}\norm{A_1-A_2}_{L^2(\Omega)^3}\leq C\ln\left(3+\norm{\Lambda_{A_1,q_1}-\Lambda_{A_2,q_2}}_{\mathcal B(H^{\frac{3}{2}}(\partial\Omega),L^2(\partial\Omega))}^{-1}\right)^{-r_1},\ee
with $r_1>0$ depending only on $s$.
For this purpose, let us fix $\xi\in\R^3\setminus\{0\}$ and consider $\eta_1,\eta_2\in\mathbb S^2$ such that $\{\xi/|\xi|,\eta_1,\eta_2\}$ is an orthonormal basis of $\R^3$. Using the notation of the previous section, we deduce that
$$\mathcal F(A)(\xi)=\frac{(\mathcal F(A)(\xi)\cdot\xi)\xi}{|\xi|^2}+(\mathcal F(A)(\xi)\cdot\eta_1)\eta_1+(\mathcal F(A)(\xi)\cdot\eta_2)\eta_2.$$
However, from condition \eqref{t1a}, we deduce that $A\in H^1(\R^3)^3$ and condition \eqref{t2a} implies that
 $$\mathcal F(A)(\xi)\cdot\xi=-i\mathcal F(\textrm{div}(A))(\xi)=-i\int_\Omega [\textrm{div}(A_1)-\textrm{div}(A_2)]e^{-ix\cdot\xi}dx=0.$$
Thus, we have 
$$\mathcal F(A)(\xi)=(\mathcal F(A)(\xi)\cdot\eta_1)\eta_1+(\mathcal F(A)(\xi)\cdot\eta_2)\eta_2$$
and applying \eqref{t1dd}, we deduce that
$$\abs{\mathcal F(A)(\xi)}\leq C\left(1+\frac{|\xi'|}{|\xi_3|}\right)^2(1+|\xi|^2)\left[\rho^{-\frac{s}{4}}+\norm{\Lambda_{A_1,q_1}- \Lambda_{A_2,q_2}}_{\mathcal B(H^{\frac{3}{2}}(\partial\Omega),L^2(\partial\Omega))}e^{c\rho}\right].$$
Combining this estimate with the arguments used at the end of the proof of Theorem \ref{t1} we deduce \eqref{t2d}.

Applying estimate \eqref{l3a} (see Lemma \ref{l3} in the Appendix) and \eqref{t2d}, we obtain
\bel{t2e}\begin{aligned}\norm{A_1-A_2}_{L^\infty(\Omega)^3}&\leq C\norm{A_1-A_2}_{W^{1,\infty}(\Omega)^3}^{\frac{3}{5}}\norm{A_1-A_2}_{L^2(\Omega)^3}^{\frac{2}{5}}\\
\ &\leq C(2M)^{\frac{3}{5}} \norm{A_1-A_2}_{L^2(\Omega)^3}^{\frac{2}{5}}\\
\ &\leq C\ln\left(3+\norm{\Lambda_{A_1,q_1}-\Lambda_{A_2,q_2}}_{\mathcal B(H^{\frac{3}{2}}(\partial\Omega),L^2(\partial\Omega))}^{-1}\right)^{-r_2},\end{aligned}\ee
with $r_2>0$ depending only on $s$.
Using the above estimate, we will now complete the proof of Theorem \ref{t2}. For this purpose, applying \eqref{tt1aa} and the estimates \eqref{l1a}, \eqref{l1b}, we obtain

\bel{t2f}\begin{aligned}&\abs{\int_{\R^3}q(x)\psi\left(\rho^{-\frac{1}{4}}x_3\right)^2\exp\left(\Phi_1+\overline{\Phi_2}\right)e^{-ix\cdot\xi}dx}\\
&\leq C\left(1+\frac{|\xi'|}{|\xi_3|}\right)^2(1+|\xi|^2)\left[\rho\ln\left(3+\norm{\Lambda_{A_1,q_1}-\Lambda_{A_2,q_2}}_{\mathcal B(H^{\frac{3}{2}}(\partial\Omega),L^2(\partial\Omega))}^{-1}\right)^{-r_2}\right.\\
&\left.\ \ \ +\rho^{-\frac{1}{4}}+\norm{\Lambda_{A_1,q_1}- \Lambda_{A_2,q_2}}_{\mathcal B(H^{\frac{3}{2}}(\partial\Omega),L^2(\partial\Omega))}e^{c\rho}\right].\end{aligned}\ee
Recalling that
$$\Phi_1+\overline{\Phi_2}=\frac{-i}{2\pi} \int_{\R^2} \frac{(\tilde{\theta}+i\eta)\cdot (A_1-A_2)(x-s_1\tilde{\theta}-s_2\eta)}{s_1+is_2}ds_1ds_2$$
and repeating the arguments used in Lemma \ref{l1}, we deduce that
$$\norm{\Phi_1+\overline{\Phi_2}}_{L^\infty(\Omega)}\leq C \norm{A_1-A_2}_{L^\infty(\Omega)^3}|(\eta_1,\eta_2)|^{-1}\leq C\left(1+\frac{|\xi'|}{|\xi_3|}\right)\norm{A_1-A_2}_{L^\infty(\Omega)^3}.$$
Moreover, applying the mean value theorem, we obtain
$$\left|\exp\left(\Phi_1+\overline{\Phi_2}\right)-1\right|\leq e^{c\|A_1-A_2\|_{L^\infty(\Omega)^3}}\|\Phi_1+\overline{\Phi_2}\|_{L^\infty(\Omega)}\leq e^{2cM}C \left(1+\frac{|\xi'|}{|\xi_3|}\right) \|A_1-A_2\|_{L^\infty(\Omega)^3}.$$
Combining this with \eqref{t2e}, we obtain
\bel{t2ee}\|\exp\left(\Phi_1+\overline{\Phi_2}\right)-1\|_{L^\infty(\Omega)}\leq C\left(1+\frac{|\xi'|}{|\xi_3|}\right)\ln\left(3+\norm{\Lambda_{A_1,q_1}-\Lambda_{A_2,q_2}}_{\mathcal B(H^{\frac{3}{2}}(\partial\Omega),L^2(\partial\Omega))}^{-1}\right)^{-r_2}.\ee
By inserting $\displaystyle\int_{\R^3} q(x)\psi\left(\rho^{-\frac{1}{4}}x_3\right)^2  e^{-i\xi\cdot x}  \, dx $, we get 
\begin{align*}
&\int_{\R^3}q(x)\psi\left(\rho^{-\frac{1}{4}}x_3\right)^2\exp\left(\Phi_1+\overline{\Phi_2}\right)e^{-ix\cdot\xi}dx \\ &= \int_{\R^3} q(x)\psi\left(\rho^{-\frac{1}{4}}x_3\right)^2  e^{-i\xi\cdot x}  \, dx + \int_{\R^3} q(x)  \psi^2\big( \rho^{-\frac{1}{4}}x_3 \big)\Big(\exp\left(\Phi_1+\overline{\Phi_2}\right) - 1\Big) e^{-i\xi\cdot x}  \, dx.
\end{align*}
It follows that
$$\begin{aligned}&\abs{\displaystyle\int_{\R^3} q(x)\psi\left(\rho^{-\frac{1}{4}}x_3\right)^2  e^{-i\xi\cdot x}  \, dx}\\
&\leq \abs{\int_{\R^3}q(x)\psi\left(\rho^{-\frac{1}{4}}x_3\right)^2\exp\left(\Phi_1+\overline{\Phi_2}\right)e^{-ix\cdot\xi}dx}+\norm{q}_{L^1(\Omega)}\|\exp\left(\Phi_1+\overline{\Phi_2}\right)-1\|_{L^\infty(\Omega)}\\
&\leq C\left(1+\frac{|\xi'|}{|\xi_3|}\right)^2(1+|\xi|^2)\left[\rho\ln\left(3+\norm{\Lambda_{A_1,q_1}-\Lambda_{A_2,q_2}}_{\mathcal B(H^{\frac{3}{2}}(\partial\Omega),L^2(\partial\Omega))}^{-1}\right)^{-r_2}\right.\\
&\left.\ \ \ +\rho^{-\frac{1}{4}}+\norm{\Lambda_{A_1,q_1}- \Lambda_{A_2,q_2}}_{\mathcal B(H^{\frac{3}{2}}(\partial\Omega),L^2(\partial\Omega))}e^{c\rho}\right].\end{aligned}$$
Combining this estimate with the arguments used in the proof Theorem \ref{t1} and in \cite[Theorem 1.1]{So}, one can check that the estimate \eqref{t2b} holds true.

\section{Stability results from measurements on some subset of the boundary}
This section is devoted to the proof of Theorem \ref{st3} and Theorem \ref{st4} by using an approach inspired by Ben Joud in \cite{B}. We will only prove Theorem \ref{st3} and we refer the reader to \cite[Theorem 1.2]{So} for the proof of Theorem \ref{st4}. For this purpose, for $j = 1, 2$, we fix $A_j \in \mathcal{A}(M, A_0, \mathcal{O}_0)$ and $q_j \in \mathcal{Q}(M, q_0, \mathcal{O}_0)$ and we consider again CGO solutions taking the form
$$u_1(x',x_3)=e^{\rho \theta\cdot x'}\left(\psi\left(\rho^{-\frac{1}{4}}x_3\right)b_1e^{i\rho x\cdot\eta-i\xi\cdot x}+w_{1,\rho}(x',x_3)\right),\quad x'\in\omega,\ x_3\in\R,$$
$$ u_2(x',x_3)=e^{-\rho \theta\cdot x'}\left(\psi\left(\rho^{-\frac{1}{4}}x_3\right)b_2 e^{i\rho x\cdot\eta}+w_{2,\rho}(x',x_3)\right),\quad x'\in\omega,\ x_3\in\R,$$
where $w_{j,\rho}\in H^2(\Omega)$ satisfies the decay property 
$$\rho^{-1}\norm{w_{1,\rho}}_{H^2(\Omega)}+\norm{w_{1,\rho}}_{H^1(\Omega)}+\rho\norm{w_{1,\rho}}_{L^2(\Omega)}\leq C(|\xi|^2+1)\left(1+\frac{|\xi'|}{|\xi_3|}\right)\rho^{\frac{7}{8}},$$
$$\rho^{-1}\norm{w_{2,\rho}}_{H^2(\Omega)}+\norm{w_{2,\rho}}_{H^1(\Omega)}+\rho\norm{w_{2,\rho}}_{L^2(\Omega)}\leq C\left(1+\frac{|\xi'|}{|\xi_3|}\right)\rho^{\frac{7}{8}}.$$
In view of Lemma \ref{lem3}, we have \begin{equation}\label{prp}
\Vert u_j \Vert_{H^2(\Omega)} \leqslant C  e^{(D+1)\rho} \left(1+\frac{|\xi'|}{|\xi_3|}\right) \Big(  1+|\xi|^2 \Big) ; \qquad j=1,2
\end{equation}
with $D:= \underset{x' \in \overline{\omega}}{sup} \vert x'\vert$. \\
We recall also that since $q:= q_1-q_2 = 0 $ in $\mathcal{O}_{0} $, we can extend $q$ to $H^1(\R^3)$  by assigning it the value $0$ outside of $\Omega$ and we denote by $q$ this extension. In this part, We need to set $\mathcal{W}_j$ ; $j=1,2,3$ such that
$$\overline{\mathcal{W}}_{j+1} \subset \mathcal{W}_j, \quad \overline{\mathcal{W}}_j \subset \mathcal{W}_0 \quad \text{and} \quad \partial \omega \subset \partial \mathcal{W}_j. $$ 
Let $\mathcal{O}_j = \mathcal{W}_j \times \R $ for $j=0,1,2,3 $.
The main idea of the proofs of Theorem \ref{st3} and Theorem \ref{st4} is to combine the estimate of the Fourier transform of $dA$ and $q$ with the weak unique continuation property which is given in the following lemma.
\begin{lemma}\label{UCP}
Let $A_1 \in \mathcal{C}^2( \overline{\Omega })$, $q_1 \in L^\infty (\Omega)$ and $M>0$ such that $\Vert q \Vert_{L^\infty (\Omega)} \leqslant M$ and let $w\in H^2(\Omega)$ solve
\begin{equation}\label{eq11}
\left\lbrace
\begin{array}{l}
\text{$(-\Delta_{A_1} + q_1)w(x) = F(x) \quad \quad \quad \, \quad \; $ in $ \Omega$, } \\
\text{$w=0 \quad \quad \quad  \quad \quad \quad \quad \quad \quad \quad  \quad \quad $ on $ \partial \Omega, $ }
\end{array}\right.
\end{equation}
where $F \in L^2(\Omega)$. Then, there exist positive constants $C$, $\alpha_1$, $\alpha_2$ and $\lambda_0$ such that we have the following estimate
\begin{equation}\label{eq12}
\Vert w \Vert_{H^1(\mathcal{O}_2 \backslash \mathcal{O}_3 )} \leqslant C \Big( e^{-\lambda \alpha_1} \Vert w \Vert_{H^2(\Omega)} + e^{\lambda \alpha_2} \Big( \Vert \partial_\nu w \Vert_{L^{2}(\Gamma_0 \times \R)} + \Vert F \Vert_{L^2(\mathcal{O}_0)} \Big) \Big)
\end{equation}
for any $\lambda \geqslant \lambda_0$. Here, the constants $C$, $\alpha_1$ and $\alpha_2$ depend on $\Omega$, $M$, $\lambda_0$, $\mathcal{O}_j$ and they are independent of $A_1$, $q_1$, $F$, $w$ and $\lambda$.
\end{lemma}
\subsection{Recovery of the magnetic field}
\begin{proof}[Proof of Theorem $\ref{st3}$]
Let $w \in H^2(\Omega)$ be the solution of
\begin{equation}\label{41}
\left\lbrace
\begin{array}{l}
\text{$-\Delta_{A_1} w + q_1  w= 0 \quad \, $ in $ \Omega$ ,} \\
\text{$w = u_2:=h \quad \quad  \qquad  $ on $ \partial \Omega $.}
\end{array}\right.
\end{equation}
Then, $u=w-u_2$ solves 
\begin{equation}\label{42}
\left\lbrace
\begin{array}{l}
\text{$-\Delta_{A_1} u + q_1  u = 2i A(x) \cdot \nabla  u_2 + V(x) u_2 (x) \quad \, $ in $ \Omega$ ,} \\
\text{$u = 0 \qquad  \qquad \qquad \qquad \qquad \qquad \qquad  \qquad \quad \; $ on $ \partial \Omega $,}
\end{array}\right.
\end{equation}
where $V(x) = i \textrm{div} (A) - \overset{\sim}{q}(x) $. Let $\Theta$ be a cut-off function satisfying $0\leqslant \Theta \leqslant 1$, $\Theta \in \mathcal{C}^\infty (\R^2)$ and 
\begin{equation}\label{eq43}
\Theta(x')=\left\lbrace
\begin{array}{ll}
1 & \mbox{in $\omega \backslash \mathcal{W}_2 $,}\\
0 & \mbox{in $\mathcal{W}_3$.}
\end{array}
\right.
\end{equation}
We set  $$\overset{\sim }{u}(x',x_3) = \Theta (x') u(x',x_3), \quad x' \in \omega, \, x_3 \in \R .$$ We remark that $\overset{\sim }{u}$ solves 
$$\left\lbrace
\begin{array}{l}
\text{$(-\Delta_{A_1} + q_1)\overset{\sim }{u}(x',x_3) =  2i \Theta(x')A(x) \cdot \nabla  u_2 + \Theta(x') V(x) u_2 (x)+P_1(x',D)u(x)$ \qquad in $\Omega $, } \\
\text{$\overset{\sim }{u}=0 \qquad \qquad \qquad \qquad \quad  \,\qquad \qquad \qquad \qquad \qquad \quad \qquad \qquad  \quad \quad \qquad \qquad \qquad  \quad \;    $ on $ \partial \Omega$, } 
\end{array}\right.$$
with $P_1(x',D)$ is given by
$$ P_1(x',D)u  = -[\Delta',\Theta]u - i A_1 [\nabla,\tilde{\Theta}]u - i [\nabla,\tilde{\Theta}]A_1 u,$$ 
where $\nabla' = (\partial_{x_1} , \partial_{x_2} )^{T}$, $\Delta'=\partial_{x_1}^2 + \partial_{x_2}^2$and
$$\tilde{\Theta}(x',x_3)=\Theta(x'),\quad x'\in\omega,\ x_3\in\R.$$
 Moreover, for an arbitrary $\overset{\sim}{v} \in H^2(\Omega)$, an integration by parts leads to
$$\int_\Omega (-\Delta_{A_1} + q_1)\overset{\sim }{u}(x)\overline{\overset{\sim}{v}(x)} \, dx = \int_\Omega \overset{\sim }{u}(x)\overline{(-\Delta_{A_1} + q_1)\overset{\sim}{v}(x) } \, dx .$$
On the other hand, we have: 
\begin{equation}\label{44}
\int_\Omega (-\Delta_{A_1} + q_1)\overset{\sim }{u}(x)\overline{\overset{\sim}{v}(x)} \, dx = \int_\R\int_\omega \big( 2i \Theta(x')A(x) \cdot \nabla  u_2 + \Theta(x') V(x) u_2 (x) +P_1(x',D)u(x) \big)\overline{\overset{\sim}{v}(x)} \, dx' \, dx_3.
\end{equation}
Choosing $\overset{\sim}{v} = u_1 $, we have $(-\Delta_{A_1} + q_1) \overset{\sim}{v} = 0$ in $\Omega$ and by the fact that $A=0$ and $q=0$ in $\mathcal{O}_0$, we get
\begin{multline}\label{45}
i \int_\Omega div (A u_2 )\overline{ u_1(x) }  \, dx + i \int_\Omega  A(x) \cdot \overline{ u_1(x) } \nabla u_2 \, dx - \int_{\Omega } ( A_1^2 - A_2^2 ) u_2 \overline{ u_1(x) } \, dx - \int_{\Omega } q u_2 \overline{ u_1(x) } \, dx \\= - \int_\Omega P_1(x',D)u(x)\overline{ u_1(x) } \, dx .
\end{multline}
By integrating by parts and using the fact that $A=0$ and $q=0$ in $\mathcal{O}_0$, we can easily obtain
\begin{equation}\label{455}
i \int_\Omega  A(x) \cdot [\overline{ u_1(x) } \nabla u_2 -u_2 \nabla \overline{ u_1(x) }] \, dx = \int_{\Omega } ( A_1^2 - A_2^2 ) u_2 \overline{ u_1(x) } \, dx + \int_{\Omega } q u_2 \overline{ u_1(x) } \, dx - \int_\Omega P_1(x',D)u(x)\overline{ u_1(x) } \, dx .
\end{equation}
Furthermore, using the fact that $P_1(x',D)u$ is supported on $ \overline{\mathcal{O}_2} \backslash\mathcal{O}_3 $, we find 
\begin{align*}
\int_\Omega \vert P_1(x',D)u \overline{ u_1(x) } \vert \, dx & \leqslant \Vert u \Vert_{H^1(\mathcal{O}_2 \backslash\mathcal{O}_3)}\Vert \overline{ u_1(x) } \Vert_{L^2(\Omega)} \\ & \leqslant C \left(1+\frac{|\xi'|}{|\xi_3|}\right) \Big(  1+|\xi|^2 \Big) e^{\rho(D+1)} \Vert u \Vert_{H^1(\mathcal{O}_2 \backslash\mathcal{O}_3)}.
\end{align*}
In a similar way to Theorem \ref{t1}, using \eqref{455}, we obtain
$$\abs{\int_{\R^3}(A\cdot(\tilde{\theta}+i\eta))e^\Phi e^{-ix\cdot\xi}dx}\leq C\left(1+\frac{|\xi'|}{|\xi_3|}\right)^2(1+|\xi|^2)\left[\rho^{-\frac{s}{4}}+e^{\rho(D+1)} \Vert u \Vert_{H^1(\mathcal{O}_2 \backslash\mathcal{O}_3)}\right] .$$
Applying \cite[Lemma 4.1]{Ki5} and following the same steps of the proof of Theorem \ref{t1}, we deduce that
$$\abs{\mathcal F(\pd_{x_k}a_j-\pd_{x_j}a_k)(\xi)}\leq C\left(1+\frac{|\xi'|}{|\xi_3|}\right)^2(1+|\xi|^3)\left[\rho^{-\frac{s}{4}}+e^{\rho(D+1)} \Vert u \Vert_{H^1(\mathcal{O}_2 \backslash\mathcal{O}_3)}\right].$$
Combining this last estimate with \eqref{eq12}, we get
$$\begin{aligned}
&\abs{\mathcal F(\pd_{x_k}a_j-\pd_{x_j}a_k)(\xi)} \\
&\leq C\left(1+\frac{|\xi'|}{|\xi_3|}\right)^2(1+|\xi|^3)\left[\rho^{-\frac{s}{4}}+e^{\rho(D+1)} \Big( e^{-\lambda \alpha_1} \Vert u \Vert_{H^2(\Omega)} + e^{\lambda \alpha_2}  \left(\Vert \partial_\nu u \Vert_{L^{2}(\Gamma_0 \times \R)}+\norm{(\Delta_{A_1}+q_1)u}_{L^2(\mathcal{O}_0)}\right) \Big) \right].\end{aligned}
$$
Since $\partial_\nu u = (\Lambda'_{A_1,q_1}- \Lambda'_{A_2,q_2} ) (h)$, where $h$ is given by \eqref{41}, we have 
\begin{align*}
\Vert \partial_\nu u \Vert_{L^{2}(\Gamma_0 \times \R)} & \leqslant C \Vert \Lambda'_{A_1,q_1}- \Lambda'_{A_2,q_2} \Vert_{\mathcal{B}(H^{\frac{3}{2}} (\partial \Omega) , H^{\frac{1}{2}} (\Gamma_0 \times R ))} \Vert h\Vert_{H^{\frac{3}{2}}(\partial \Omega)} \\
&  \leqslant C  e^{(D+1)\rho} \left(1+\frac{|\xi'|}{|\xi_3|}\right) \Big(  1+|\xi|^2 \Big) \Vert \Lambda'_{A_1,q_1}- \Lambda'_{A_2,q_2} \Vert_{\mathcal{B}(H^{\frac{3}{2}} (\partial \Omega) , H^{\frac{1}{2}} (\Gamma_0 \times R ))} .
\end{align*}
Moreover, since $A_j \in \mathcal{A}(M, A_0 , \mathcal{O}_0)$ and $q_j \in \mathcal{Q}(M, q_0 , \mathcal{O}_0)$, $j=1,2$, we have $(\Delta_{A_1}+q_1)u= 0$ on $\mathcal{O}_0$ and it follows by $\eqref{prp}$ \begin{multline}\label{51}
\abs{\mathcal F(\pd_{x_k}a_j-\pd_{x_j}a_k)(\xi)} \\ \leqslant C \left(1+\frac{|\xi'|}{|\xi_3|}\right)^3(1+|\xi|^3)^2 \left( \rho^{-\frac{s}{4}}+e^{2\rho(D+1)-\lambda \alpha_1} + e^{2\rho(D+1)+\lambda \alpha_2}  \Vert \Lambda'_{A_1,q_1}- \Lambda'_{A_2,q_2} \Vert_{\mathcal{B}(H^{\frac{3}{2}} (\partial \Omega) , H^{\frac{1}{2}} (\Gamma_0 \times R ))} \right) .
\end{multline}
Let $D':= D+1$ and $\lambda = \tau \rho $. Choosing $\tau $ sufficiently large, it becomes easy to find constants $\alpha_3$ and $\alpha_4$ such that 
\begin{equation}\label{52}
e^{2D'\rho-\lambda \alpha_1} =  e^{\rho(2D' -\tau \alpha_1)} \leqslant e^{-\alpha_3 \rho} \quad \text{ and } \quad e^{2D'\rho+\lambda \alpha_2}=e^{\rho(2D'  + \tau \alpha_2)} \leqslant e^{\alpha_4 \rho} .
\end{equation}
Combining $\eqref{51}$ and $\eqref{52}$, we conclude that 
\begin{multline}\label{53}
\abs{\mathcal F(\pd_{x_k}a_j-\pd_{x_j}a_k)(\xi)} \\ \leqslant C \left(1+\frac{|\xi'|}{|\xi_3|}\right)^3(1+|\xi|^3)^2 \left( \rho^{-\frac{s}{4}}+e^{- \alpha_3 \rho} + e^{\alpha_4 \rho }  \Vert \Lambda'_{A_1,q_1}- \Lambda'_{A_2,q_2} \Vert_{\mathcal{B}(H^{\frac{3}{2}} (\partial \Omega) , H^{\frac{1}{2}} (\Gamma_0 \times R ))} \right) \\ \leq C \left(1+\frac{|\xi'|}{|\xi_3|}\right)^3(1+|\xi|^3)^2 \left( \rho^{-\frac{s}{4}}+ e^{\alpha_4 \rho }  \Vert \Lambda'_{A_1,q_1}- \Lambda'_{A_2,q_2} \Vert_{\mathcal{B}(H^{\frac{3}{2}} (\partial \Omega) , H^{\frac{1}{2}} (\Gamma_0 \times R ))} \right).
\end{multline} 
From now on we fix $R>1$, $\gamma':= \norm{\Lambda'_{A_1,q_1}- \Lambda'_{A_2,q_2}}_{\mathcal{B}(H^{\frac{3}{2}} (\partial \Omega) , H^{\frac{1}{2}} (\Gamma_0 \times R ))}$ and we consider the set
$$D_R=\{\xi\in B_R:\ |\xi_3|\geq R^{-4}\}.$$
We obtain the estimate
$$\abs{\mathcal F(\pd_{x_k}a_j-\pd_{x_j}a_k)(\xi)}\leq C(R^{21}\rho^{-\frac{s}{4}}+R^{21}\gamma' e^{\alpha_4 \rho }),\quad \xi\in D_R.$$
It follows that
\bel{t1eee}\int_{D_R}\abs{\mathcal F(\pd_{x_k}a_j-\pd_{x_j}a_k)(\xi)}^2d\xi \leq C(R^{45}\rho^{-\frac{s}{2}}+R^{45}\gamma'^2 e^{2\alpha_4\rho}).\ee
Combining \eqref{t1eee}-\eqref{t1gg}, we get
$$\int_{\R^3}\abs{\mathcal F(\pd_{x_k}a_j-\pd_{x_j}a_k)(\xi)}^2d\xi\leq C(R^{-2}+R^{45}\rho^{-\frac{s}{2}}+R^{45}\gamma'^2 e^{2\alpha_4\rho})$$
and by Plancherel formula, it follows
$$\norm{dA}_{L^2(\Omega)}\leq C(R^{-1}+R^{45/2}\rho^{-\frac{s}{4}}+R^{45/2}\gamma' e^{\alpha_4\rho}).$$
Choosing $R=\rho^{\frac{s}{94}}$, we get
\begin{align}\label{3332}
\norm{dA}_{L^2(\Omega)}&\leq C(\rho^{-\frac{s}{94}}+(\rho^{-\frac{s}{94}})^{45/2}\gamma' e^{\alpha_4\rho}) \nonumber \\ &\leq C(\rho^{-\frac{s}{94}}+\gamma' e^{(\alpha_4+1)\rho}).
\end{align}
Then, repeating the arguments used at the end of the proof of Theorem \ref{t1}, we can deduce $\eqref{st3}$ from $\eqref{3332}$.
\end{proof}

\setcounter{equation}{0}

\appendix
\section{ }
In this appendix, we consider the following interpolation result. 
\begin{lemma} \label{l3} Let $h\in W^{1,\infty}(\Omega)\cap L^2(\Omega)$. Then there exists $C>0$ depending only on $\Omega$ such that 
\bel{l3a}\norm{h}_{L^\infty(\Omega)}\leq C\norm{h}_{W^{1,\infty}(\Omega)}^{\frac{3}{5}}\norm{h}_{L^2(\Omega)}^{\frac{2}{5}}.\ee
\end{lemma}
\begin{proof} Let us first observe that this result is well known for $\Omega$ bounded (see e.g. \cite[Lemma Appendix B.1.]{CK}), but we have not find any proof of it for unbounded domains. For this reason, we decided do give the full proof of this result. Let us fix $\psi\in\mathcal C^\infty_0(-2,2)$ satisfying $\psi=1$ on $[-1,1]$ and fix $y_3\in \R$. Denote also by $\mathcal O$ a smooth open bounded subset of $\Omega$ such that $\omega\times[-2,2]\subset\mathcal O$. Now let us consider the function $h_{y_3}:x=(x',x_3)\mapsto \psi(y_3+x_3)h(x',x_3+y_3)$. Applying \eqref{l3a} for $\Omega=\mathcal O$ and $h=h_{y_3}$, we obtain that
\bel{l3b}\norm{h_{y_3}}_{L^\infty(\mathcal O)}\leq C\norm{h_{y_3}}_{W^{1,\infty}(\mathcal O)}^{\frac{3}{5}}\norm{h_{y_3}}_{L^2(\mathcal O)}^{\frac{2}{5}},\ee
with $C>0$ depending only on $\mathcal O$. In the same way, we have
$$\norm{h_{y_3}}_{W^{1,\infty}(\mathcal O)}\leq \norm{\psi h}_{W^{1,\infty}(\Omega)}\leq \norm{\psi}_{W^{1,\infty}(\R)}\norm{ h}_{W^{1,\infty}(\Omega)},$$
$$\norm{h_{y_3}}_{L^2(\mathcal O)}^2\leq \norm{\psi}_{L^\infty(\R)}^2\int_\omega\int_\R|h(x',x_3+y_3)|^2dx'dx_3=\norm{\psi}_{L^\infty(\R)}^2\norm{h}_{L^2(\Omega)}^2.$$
Combining these two estimates with \eqref{l3b}, we obtain
$$\norm{h_{y_3}}_{L^\infty(\mathcal O)}\leq C\norm{\psi}_{W^{1,\infty}(\R)}\norm{h}_{W^{1,\infty}(\Omega)}^{\frac{3}{5}}\norm{h}_{L^2(\Omega)}^{\frac{2}{5}}.$$
Since the right hand side of the above identity is independent of $y_3\in\R$, we can take the sup with respect to $y_3\in\R$ in order to deduce that
$$\norm{h}_{L^\infty(\Omega)}\leq\sup_{y_3\in\R}\norm{h_{y_3}}_{L^\infty(\mathcal O)}\leq C\norm{\psi}_{W^{1,\infty}(\R)}\norm{h}_{W^{1,\infty}(\Omega)}^{\frac{3}{5}}\norm{h}_{L^2(\Omega)}^{\frac{2}{5}}.$$
This estimate clearly implies \eqref{l3a}.\end{proof}
\setcounter{equation}{0}

\section{Carleman's estimate}
The main goal of this appendix is to prove a Carleman's estimate for the magnetic Schr\"odinger operator $-\Delta_A + q $ in an infinite cylindrical domain in order to deduce the weak unique continuation property given in Lemma \ref{UCP}. As we are dealing with weighted inequalities, we borrow, from \cite{So} (see also \cite[Lemma 2.3]{IY}, \cite[Lemma 1.2]{Im} and \cite[Theorem 2.4]{KU1}), the following result that guarantees the existence of the weigh function.
\begin{lemma}\label{lempsi}
There exists a function $\psi_0 \in \mathcal{C}^3(\overline{\mathcal{W}}_0)$ such that
\begin{enumerate}
\item[(i)] $\psi_0 (x') > 0$ for all $x' \in \mathcal{W}_0$,
\item[(ii)] There exists $\alpha_0 > 0$ such that $\vert \nabla^{'} \psi_0 (x') \vert \geqslant \alpha_0 $ for all $x' \in \overline{\mathcal{W}}_0$,
\item[(iii)] $\partial_{\nu'} \psi_0 (x') \leqslant 0 $ for all $x' \in \partial \mathcal{W}_0 \backslash \Gamma_0$,
\item[(iv)] $\psi_0 (x')= 0$ for all $x' \in \partial \mathcal{W}_0 \backslash \Gamma_0$.
\end{enumerate}
\end{lemma}
Here $\nabla^{'}$ denotes the gradient with respect to $x'\in \R^2$ and $\partial_{\nu'}$ is the normal derivative with respect to $\partial \mathcal{W}_0$, that is $\partial_{\nu'} := \nu' \cdot \nabla^{'}$ where $\nu'$ stands for the outward normal vector to $\partial  \mathcal{W}_0$.\\
Thus, putting $\psi (x) = \psi(x',x_3):=\psi_0(x')$ for all $x=(x',x_3)\in \overline{\mathcal{O}_0}$, it is apparent that the function $\psi \in \mathcal{C}^3(\overline{\mathcal{O}_0})$ satisfies the three following conditions:
\begin{itemize}
\item[(C1)] $\psi(x) > 0$, \quad $x \in \mathcal{O}_0$,
\item[(C2)] $\vert \nabla \psi (x) \vert \geqslant \alpha_0 $ for all $x \in \overline{\mathcal{O}_0}$,
\item[(C3)] $\partial_{\nu} \psi(x) \leqslant 0 $ for all $x \in \partial \mathcal{O}_0 \backslash (\Gamma_0 \times \R)$,
\item[(C4)] $\psi(x)=0$ for all $x \in \Gamma^{\sharp} \times \R$.
\end{itemize}
Here $\nu$ is the outward unit normal vector to the boundary $\partial \mathcal{O}_0$. Evidently $\nu=(\nu',0)$ so we have $ \partial_{\nu} \psi = \partial_{\nu'} \psi_0 $ as the function $\psi$ does not depend on $x_3$.\\
Next, for $\beta \in (0,+\infty)$, we introduce the following weigh function 
\begin{equation}\label{eqphi}
\varphi(x)=\varphi(x')=e^{\beta \psi(x)}; \quad x \in \mathcal{O}_0 .
\end{equation}
Then $\varphi$ satisfies some properties given in the following result.
\begin{lemma}\label{phii}
There exists a constant $\beta_0 \in (0, +\infty)$ depending only on $\psi$ such that the following statements hold uniformly in $\mathcal{O}_0$ for all $\beta \in [\beta_0 , +\infty )$.
\begin{enumerate}
\item[(a)] $\vert \nabla \varphi \vert \geqslant \alpha:= \beta_0 \alpha_0 $,
\item[(b)] $\nabla \vert \nabla \varphi \vert^2 \cdot \nabla \varphi \geqslant C_0 \beta \vert \nabla \varphi \vert^3$,
\item[(c)] $\mathcal{H}(\varphi)\xi \cdot \xi + C_1 \beta \vert \nabla \varphi \vert \vert \xi \vert^2 \geqslant 0 \quad ; \quad \xi \in \R^3$,
\item[(d)] $\vert \Delta \vert \nabla \varphi \vert \vert \leqslant C_2 \vert \nabla \varphi \vert^3$,
\item[(e)] $\Delta \varphi \geqslant 0$.
\end{enumerate}
Here, $C_0$, $C_1$ and $C_2$ are positive constants depending only on $\psi$ and $\alpha_0$ and $\mathcal{H}(\varphi)$ denotes the Hessian matrix of $\varphi$ with respect to $x\in \mathcal{O}_0$.
\end{lemma}
Now, we may state the following Carleman's estimate for the operator $\Delta_A + q$.
\begin{theorem}
Let $u \in H_0^1(\mathcal{O}_0) \cap H^2(\mathcal{O}_0)$, $M_1,\, M_2>0$ and let $A\in W^{1,\infty}$ and $q \in L^{\infty}(\Omega)$ satisfy $\Vert A \Vert_{W^{1,\infty} (\mathcal{O}_0)} \leqslant M_1$ and $\Vert q  \Vert_{L^{\infty} (\mathcal{O}_0)} \leqslant M_2$. Then, there exists $\beta_0 \in (0, +\infty)$ such that for every $\beta \geqslant \beta_0$, there is $\lambda_0=\lambda_0(\beta) \in (0, +\infty)$ depending only on $\beta$, $\alpha_0$, $\mathcal{O}_0$, $M_1$, $M_2$ and $\Gamma_0$, such that the estimate 
\begin{equation}\label{eq2}
\lambda \int_{\mathcal{O}_0} e^{2\lambda\varphi} \big( \lambda^2 \vert u \vert^2 + \vert \nabla u \vert^2 \big) \, dx \\
\leqslant C \Big( \int_{\mathcal{O}_0} e^{2\lambda\varphi} \vert (\Delta_A + q)u \vert^2 \, dx + \lambda \int_{\Gamma_0 \times \R} e^{2\lambda\varphi} \big\vert \partial_\nu u \big\vert^2  \, d\sigma_x \Big).
\end{equation}
holds for all $\lambda \geqslant \lambda_0$ and some positive constant $C$ that depends only on $\alpha_0$, $\omega$, $\Gamma_0$, $\beta$ and $\lambda_0$.
\end{theorem}
\begin{proof}
For the proof, we can simply show the following inequality 
\begin{equation}\label{eq3}
\lambda \int_{\mathcal{O}_0} e^{2\lambda\varphi} \big( \lambda^2 \vert u \vert^2 + \vert \nabla u \vert^2 \big) \, dx \\
\leqslant C \Big( \int_{\mathcal{O}_0} e^{2\lambda\varphi} \vert \Delta u \vert^2 \, dx + \lambda \int_{\Gamma_0 \times \R } e^{2\lambda\varphi} \big\vert \partial_\nu u \big\vert^2  \, d\sigma_x \Big).
\end{equation}
In fact, we have
$$\Delta_A + q = \Delta + P_0 ,$$ 
with $P_0$ is a first order operator given by
$$P_0 = 2iA\cdot\nabla + i\textrm{div}(A) - A\cdot A + q.$$ 
As $\Vert A \Vert_{W^{1,\infty} (\mathcal{O}_0)} \leqslant M_1$ and $\Vert q  \Vert_{L^{\infty} (\mathcal{O}_0)} \leqslant M_2$, we get:
$$ \vert P_0 u \vert \leqslant C ( \vert u \vert + \vert \nabla u\vert ).$$
By $ (\ref{eq3})$, we get
\begin{multline*}
\lambda \int_{\mathcal{O}_0} e^{2\lambda\varphi} \big( \lambda^2 \vert u \vert^2 + \vert \nabla u \vert^2 \big) \, dx \\
\leqslant C \Big( \int_{\mathcal{O}_0} e^{2\lambda\varphi} ( \vert (\Delta_A + q) u \vert^2 + \vert P_0 u \vert^2) \, dx + \lambda \int_{\Gamma_0 \times \R} \big\vert \partial_\nu u \big\vert^2 e^{2\lambda\varphi} \, d\sigma_x \Big) \\
\leqslant C \int_{\mathcal{O}_0} e^{2\lambda\varphi} \vert (\Delta_A + q) u \vert^2 \, dx + C \int_{\mathcal{O}_0} e^{2\lambda\varphi} ( \vert u \vert^2 + \vert \nabla u\vert^2 ) \, dx + C \lambda \int_{\Gamma_0 \times \R} \big\vert \partial_\nu u \big\vert^2 e^{2\lambda\varphi} \, d\sigma_x \\
\forall \lambda \geqslant \tau_0.
\end{multline*}
Thus, we have
\begin{multline*}
\lambda^3 \big( 1-\dfrac{C}{\lambda^3} \big) \int_{\mathcal{O}_0} e^{2\lambda\varphi} \vert u \vert^2 \, dx + \lambda \big(1-\dfrac{C}{\lambda} \big) \int_{\mathcal{O}_0} e^{2\lambda\varphi} \vert \nabla u \vert^2 \, dx \\
\leqslant C \int_{\mathcal{O}_0} e^{2\lambda\varphi} \vert (\Delta_A + q)u \vert^2 \, dx + C \lambda \int_{\Gamma_0 \times \R} \big\vert \partial_\nu u \big\vert^2 e^{2\lambda\varphi} \, d\sigma_x.
\end{multline*}
Let $\lambda_0'$ such that $\forall \lambda \geqslant \lambda_0'$, $1-\dfrac{C}{\lambda^3} \geqslant \dfrac{1}{2}$ and $1-\dfrac{C}{\lambda} \geqslant \dfrac{1}{2}$. For any $ \lambda \geqslant max(\lambda_0,\lambda_0')$, we have
$$
\dfrac{1}{2}\lambda^3 \int_{\mathcal{O}_0} e^{2\lambda\varphi} \vert u \vert^2 \, dx + \dfrac{1}{2} \lambda \int_{\mathcal{O}} e^{2\lambda\varphi} \vert \nabla u \vert^2 \, dx \\
\leqslant C \int_{\mathcal{O}_0} e^{2\lambda\varphi} \vert (\Delta_A + q)u \vert^2 \, dx + C \lambda \int_{\Gamma_0 \times \R} \big\vert \partial_\nu u \big\vert^2 e^{2\lambda\varphi} \, d\sigma_x.
$$
The proof of the estimate $(\ref{eq3})$ is stated in \cite{So}. 
\end{proof}
\setcounter{equation}{0}
\section{Weak unique continuation property}
This appendix is devoted to the proof of the weak unique continuation property stated in Lemma \ref{UCP}. Let $\psi_0$ be the function defined in Lemma \ref{lempsi}.
Since $\psi_0(x')>0$ for all $x' \in \mathcal{W}_0$, there exists a constant $\kappa >0$ such that
\begin{equation} \label{eq13}
\psi_0(x') \geqslant 2 \kappa; \quad x' \in \mathcal{W}_2 \backslash \mathcal{W}_3  .
\end{equation} 
Moreover, as $\psi_0(x')=0$, $x'\in \Gamma^\sharp $, there exist $\mathcal{W}^\sharp $ a small neighborhood of $\Gamma^\sharp $ such that 
\begin{equation} \label{eq14}
\psi_0(x') \leqslant \kappa; \quad  x' \in \mathcal{W}^\sharp ,  \quad \mathcal{W}^\sharp \cap \overline{\mathcal{W}}_1 = \varnothing.
\end{equation}
Let $\overset{\sim}{\mathcal{W}}^\sharp \subset \mathcal{W}^\sharp$ be an arbitrary neighborhood of $\Gamma^\sharp $. To apply $(\ref{eq2})$, it is necessary to introduce a function $\Theta$ satisfying $0\leqslant \Theta \leqslant 1$, $\Theta \in \mathcal{C}^\infty (\R^2)$ and 
\begin{equation}\label{eq15}
\Theta(x')=\left\lbrace
\begin{array}{ll}
1 & \mbox{in $\mathcal{W}_0 \backslash \mathcal{W}^\sharp$,}\\
0 & \mbox{in $\overset{\sim}{\mathcal{W}}^\sharp$.}
\end{array}
\right.
\end{equation}
Let $w$ be a solution to $(\ref{eq11})$. Setting $$w_1(x',x_3)= \Theta(x')w(x',x_3), \quad x' \in \omega, \, x_3 \in \R,$$ 
we get
$$\left\lbrace
\begin{array}{l}
\text{$(-\Delta_{A_1 , q_1} + q_1)w_1(x) = \Theta(x)F(x)+Q_1(x,D)w \quad \quad \quad \quad $ in $ \mathcal{O}_0$, } \\
\text{$w_1=0 \quad \quad \quad  \quad \quad \,\quad \quad \quad \quad \quad \quad \quad \quad  \quad \quad \quad \quad \quad  \quad \,  $ on $ \partial \mathcal{O}_0$, } 
\end{array}\right.$$
where $Q_1(x,D)$ is a first order operator supported in $ \overline{\mathcal{W}^\sharp} \backslash \overset{\sim}{\mathcal{W}}^\sharp$ and given by
$$ Q_1(x,D)w  = -[\Delta',\Theta]w-i A_1 [\nabla,\tilde{\Theta}]w - i[\nabla,\tilde{\Theta}]A_1 w .$$ 
By applying Carleman estimate $(\ref{eq2})$ to $w_1$, we obtain
\begin{multline}\label{eq16}
\lambda \int_{\mathcal{O}_0} e^{2\lambda\varphi} \big( \lambda^2 \vert w_1 \vert^2 + \vert \nabla w_1 \vert^2 \big) \, dx \\
\leqslant C \Big( \int_{\mathcal{O}_0} e^{2\lambda\varphi} \Big( \big\vert Q_1(x,D)w \big\vert^2 + \big\vert F(x) \big\vert^2 \Big) \, dx + \lambda \int_{\Gamma_0 \times \R} \big\vert \partial_\nu w_1 \big\vert^2 e^{2\lambda\varphi} \, d\sigma_{x} \Big).
\end{multline}
Let $ \mathcal{O}^\sharp = \mathcal{W}^\sharp \times \R$ and $\overset{\sim}{\mathcal{O}}^\sharp = \overset{\sim}{\mathcal{W}}^\sharp \times \R $. Using the fact that $Q_1(x,D)$ is a first order operator supported in $ \overline{\mathcal{O}^\sharp} \backslash \overset{\sim}{\mathcal{O}}^\sharp$ and by $(\ref{eq14})$, we get
\begin{align*}
\int_{\mathcal{O}_0} e^{2\lambda\varphi} \big\vert Q_1(x,D)w \big\vert^2 \, dx & \leqslant \int_{\mathcal{O}_0} e^{2\lambda e^{\beta\psi(x)}} \big\vert Q_1(x,D)w \big\vert^2 \, dx \\
& \leqslant e^{2\lambda e^{\beta\kappa}}  \int_{\mathcal{O}^\sharp \backslash \overset{\sim}{\mathcal{O}}^\sharp} \big\vert Q_1(x,D)w \big\vert^2 \, dx \\
& \leqslant C e^{2\lambda e^{\beta\kappa}}  \int_{\mathcal{O}^\sharp \backslash \overset{\sim}{\mathcal{O}}^\sharp} \big( \vert w \vert^2 + \vert \nabla w \vert^2 \big)  \, dx.
\end{align*}
On the other hand, by using the definition of $\Theta$ given by $(\ref{eq15})$, the estimate $(\ref{eq16})$ becomes
\begin{multline*}
\lambda \int_{\mathcal{O}_0 \backslash \overset{\sim}{\mathcal{O}}^\sharp} e^{2\lambda\varphi}  \big( \lambda^2 \vert w \vert^2 + \vert \nabla w \vert^2 \big)  \, dx \leqslant  C \Big( e^{2\lambda e^{\beta\kappa}} \int_{\mathcal{O}^\sharp \backslash \overset{\sim}{\mathcal{O}}^\sharp} \big( \vert w \vert^2 + \vert \nabla w \vert^2 \big)  \, dx \\
+ \int_{\mathcal{O}_0} e^{2\lambda\varphi} \big\vert F(x) \big\vert^2 \, dx + \lambda \int_{\Gamma_0 \times \R} \big\vert \partial_{\nu} w \big\vert^2 e^{2\lambda\varphi} \, d\sigma_{x} \Big).
\end{multline*}
Moreover, by the fact that $\mathcal{O}_2 \backslash \mathcal{O}_3 \subset \mathcal{O}_0 \backslash \overset{\sim}{\mathcal{O}}^\sharp$ and by $(\ref{eq13})$, we easily obtain that
\begin{multline*}
e^{2\lambda e^{2\beta\kappa}} \lambda \int_{\mathcal{O}_2 \backslash \mathcal{O}_3 }  \big( \lambda^2 \vert w \vert^2 + \vert \nabla w \vert^2 \big)  \, dx \leqslant  C \Big( e^{2\lambda e^{\beta\kappa}} \int_{\mathcal{O}^\sharp \backslash \overset{\sim}{\mathcal{O}}^\sharp} \big( \vert w \vert^2 + \vert \nabla w \vert^2 \big)  \, dx \\
+ \int_{\mathcal{O}_0} e^{2\lambda\varphi} \big\vert F(x) \big\vert^2 \, dx + \lambda \int_{\Gamma_0 \times \R} \big\vert \partial_{\nu} w \big\vert^2 e^{2\lambda\varphi} \, d\sigma_{x} \Big).
\end{multline*}
Thus, we have
\begin{multline*}
\lambda \int_{\mathcal{O}_2 \backslash \mathcal{O}_3 }  \big( \lambda^2 \vert w \vert^2 + \vert \nabla w \vert^2 \big)  \, dx \leqslant  C \Big( e^{-2\lambda ( e^{2\beta\kappa} - e^{\beta\kappa})} \int_{\mathcal{O}^\sharp \backslash \overset{\sim}{\mathcal{O}}^\sharp} \big( \vert w \vert^2 + \vert \nabla w \vert^2 \big)  \, dx \\
+e^{2\lambda( e^{2\beta\Vert\psi_0\Vert_\infty} - e^{2\beta\kappa})} \Big( \int_{\mathcal{O}_0}  \big\vert F(x) \big\vert^2 \, dx + \lambda \int_{\Gamma_0 \times \R} \big\vert \partial_{\nu} w \big\vert^2  \, d\sigma_{x} \Big) \Big).
\end{multline*}
Let $ \alpha_1= ( e^{2\beta\kappa} - e^{\beta\kappa}) \, >0$ and $\alpha_2=( e^{2\beta\Vert\psi_0\Vert_\infty} - e^{2\beta\kappa}) \, >0$. We conclude that for any $\lambda>\lambda^*$, we have:
\begin{multline*}
\lambda \int_{\mathcal{O}_2 \backslash \mathcal{O}_3 }  \big( \lambda^2 \vert w \vert^2 + \vert \nabla w \vert^2 \big)  \, dx \leqslant  C \Big( e^{-2\lambda \alpha_1} \int_{\mathcal{O}^\sharp \backslash \overset{\sim}{\mathcal{O}}^\sharp} \big( \vert w \vert^2 + \vert \nabla w \vert^2 \big)  \, dx \\
+e^{2\lambda \alpha_2} \Big( \int_{\mathcal{O}_0}  \big\vert F(x) \big\vert^2 \, dx + \int_{\Gamma_0 \times \R} \big\vert \partial_{\nu} w \big\vert^2 \, d\sigma_{x} \Big) \Big).
\end{multline*}
Then, we have
$$ \Vert w \Vert_{H^1(\mathcal{O}_2 \backslash \mathcal{O}_3)}^2 \leqslant C \Big( e^{-2\lambda \alpha_1} \Vert w \Vert_{H^1(\Omega)}^2 + e^{2\lambda \alpha_2} \Big( \Vert F \Vert_{L^2(\mathcal{O}_0)}^2  + \big\Vert \partial_{\nu} w \big\Vert_{L^2(\Gamma_0 \times \R)}^2 \Big) \Big) $$
which completes the demonstration.

\section*{Acknowledgments}

 This work was partially supported by  the French National
Research Agency ANR (project MultiOnde) grant ANR-17-CE40-0029.


\begin{thebibliography} {[10]}
\frenchspacing \baselineskip=12 pt plus 1pt minus 1pt
\bibitem{Al} {\sc G. Alessandrini}, {\em Stable determination of conductivity by boundary measurements}, Appl.
Anal., \textbf{27} (1-3) (1988), 153-172.
\bibitem{BKS} {\sc M. Bellassoued, Y. Kian, E. Soccorsi}, {\em An inverse stability result for non compactly supported potentials by one arbitrary lateral Neumann observation},
J. Diff. Equat., \textbf{260} (2016), 7535-7562.  
 \bibitem{BKS1} {\sc M. Bellassoued, Y. Kian, E. Soccorsi}, {\em An inverse problem for the magnetic Schr\"odinger equation in infinite cylindrical domains},  Publ. Research Institute Math.   Sci.,  	\textbf{54} (2018), 679-728.
\bibitem{B} {\sc H. Ben Joud}, {\em A stability estimate for an inverse problem for the
Schr\"odinger equation in a magnetic field from partial boundary measurements}, Inverse Problems, \textbf{25} (2009) 045012.
 \bibitem{Ca}{\sc A. P. Calder\'on}, {\em On an inverse boundary value problem}, Seminar on Numerical Analysis
and its Applications to Continuum Physics, Rio de Janeiro, Sociedade Brasileira
de Matematica, (1980), 65-73.
\bibitem{CDR2} {\sc P. Caro, D. Dos Santos Ferreira, A. Ruiz}, {\em Stability estimates for the Calder\'on problem with partial data},  J. Diff. Equat., \textbf{260} (2016), 2457-2489.
\bibitem{CM}{\sc P. Caro and K. Marinov}, {\em Stability of inverse problems in an infinite slab with partial data}, Commun. Partial Diff. Eqns., \textbf{41} (2016), 683-704.
\bibitem{CP}{\sc P. Caro and  V. Pohjola}, {\em Stability Estimates for an Inverse Problem for the Magnetic Schr\"odinger Operator}, IMRN, \textbf{2015} (2015), 11083-11116.
\bibitem{CL} {\sc P.-Y. Chang and H.-H. Lin}, {\em Conductance through a single impurity in the metallic zigzag carbon nanotube}, Appl. Phys. Lett., {\bf 95} (2009), 082104.
\bibitem{Ch}{\sc M. Choulli}, {\em Une introduction aux probl\`emes inverses elliptiques et paraboliques}, Math\'ematiques et Applications, Vol. 65, Springer-Verlag, Berlin, 2009.
\bibitem{CK}{\sc M. Choulli and Y. Kian}, {\em Logarithmic stability in determining the time-dependent zero order coefficient in a parabolic equation from a partial Dirichlet-to-Neumann map. Application to the determination of a nonlinear term}, J. Math. Pures Appl., 	\textbf{114} (2018), 235-261.
\bibitem{CKS}{\sc M. Choulli, Y. Kian, E. Soccorsi}, {\em Stable determination of time-dependent scalar potential from boundary measurements in a periodic quantum waveguide}, SIAM J. Math. Anal., {\bf 47} (2015), no 6, 4536-4558.
\bibitem{CKS4}{\sc M. Choulli, Y. Kian, E. Soccorsi}, {\em Double logarithmic stability estimate in the identification of a scalar potential by a partial elliptic Dirichlet-to-Neumann map}, Bulletin of the South Ural State University, Ser. Mathematical Modelling, Programming and Computer Software (SUSU MMCS) {\bf 8} (2015), no 3, 78-95.
\bibitem{CKS2}{\sc M. Choulli, Y. Kian, E. Soccorsi},  {\em Stability result for elliptic inverse periodic coefficient problem by partial Dirichlet-to-Neumann map},    J. Spec. Theory, \textbf{8} (2) (2018), 733-768.
\bibitem{CKS3}{\sc M. Choulli, Y. Kian, E. Soccorsi},  {\em On the Calder\'on problem in periodic cylindrical domain with partial Dirichlet and Neumann data},   Math. Meth. Appl. Sci., \textbf{40} (2017), 5959-5974.
\bibitem{CS}{\sc M. Choulli and E. Soccorsi}, {\em An inverse anisotropic conductivity problem induced by twisting a homogeneous cylindrical domain}, J. Spec. Theory, \textbf{5} (2015), 295-329. 
\bibitem{EE}{\sc D. Edmunds and W. Evans}, {\em Spectral theory and differential operators}, Oxford University Press, New York, 1987.
\bibitem{FKSU}{\sc D. Dos Santos Ferreira, C. Kenig, J. Sj\"ostrand, G. Uhlmann}, {\em Determining a magnetic Schr\"odinger operator from partial Cauchy data},Comm. Math. Phys.,  \textbf{271} No. 2 (2007), 467-488. 
\bibitem{GL}{\sc N. Garofalo and F-H. Lin}, {\em Unique continuation for elliptic operators: a  geometric-variational approach}, Communications on Pure and Applied Mathematics, \textbf{40} (1987), 347-366.
\bibitem{Ha} {\sc P. H\"ahner}, {\em A periodic Faddeev-type operator}, J. Diff. Equat., {\bf 128} (1996), 300-308.
\bibitem{Ho3}{\sc L. H\"ormander}, {\em The Analysis of linear partial differential operators}, Vol III, Springer-Verlag, Berlin, Heidelberg, 1983.
\bibitem{Ik} {\sc M. Ikehata}, {\em Inverse conductivity problem in the infinite slab}, Inverse Problems, {\bf 17} (2001), 437-454.
\bibitem{Im}{\sc O. Imanuvilov}, {\em Controllability of evolution equations}, Sb. Math., \textbf{186} (1995), 186-879.
\bibitem{IY}{\sc O. Imanuvilov and M. Yamamoto}, {\em Lipschitz stability in inverse parabolic problems by the Carleman estimate}, Inverse Problems, \textbf{14} (1998), no. 5, 1229-1245.
\bibitem{KBF} {\sc C. Kane, L. Balents, M. P. A. Fisher}, {\em Coulomb Interactions and Mesoscopic Effects in Carbon Nanotubes}, Phys. Rev. Lett., {\bf 79} (1997), 5086-5089.
\bibitem{KKS} {\sc O. Kavian, Y. Kian, E. Soccorsi}, {\em Uniqueness and stability results for an inverse spectral problem in a
  periodic waveguide}, Jour. Math. Pures Appl., {\bf 104} (2015), no. 6, 1160-1189.
\bibitem{Ki1}{\sc Y. Kian}, {\em Stability of the determination of a coefficient for wave equations in an infinite waveguide}, Inverse Probl. Imaging, \textbf{8} (3) (2014),  713-732.
 \bibitem{Ki3}{\sc Y. Kian}, {\em Recovery of time-dependent damping coefficients and potentials appearing in wave equations from partial data}, SIAM J. Math. Anal., \textbf{48} (6) (2016), 4021-4046.
\bibitem{Ki4}{\sc Y. Kian}, {\em Recovery of non compactly supported coefficients of elliptic equations on an infinite waveguide},   Journal of the Institute of Mathematics of Jussieu, \textbf{19} (2020), 1573-1600.
 \bibitem{Ki5}{\sc Y. Kian}, {\em Determination of non-compactly supported electromagnetic potentials in an unbounded closed waveguide},  Revista Matem\'atica Iberoamericana, \textbf{36} (2020), 671-710.
 \bibitem{Ki6}{\sc Y. Kian}, {\em Lecture on the Calder\'on problem}, Contemporary Mathematics, \textbf{757} (2020), 1-18.
\bibitem{KPS1}{\sc Y. Kian, Q. S. Phan, E. Soccorsi}, {\em Carleman estimate for infinite cylindrical quantum domains and application to inverse problems},
Inverse Problems, {\bf 30}, 5 (2014), 055016. 
\bibitem{KPS2}{\sc Y. Kian, Q. S. Phan, E. Soccorsi}, {\em H\"older stable determination of a quantum scalar potential in unbounded cylindrical domains},
Jour. Math. Anal. Appl., {\bf 426}, 1 (2015), 194-210. 
\bibitem{KLU} {\sc K. Krupchyk, M. Lassas, G. Uhlmann}, {\em Inverse Problems with Partial Data for a Magnetic Schr\"odinger Operator in an Infinite Slab or Bounded Domain},
Comm. Math. Phys., {\bf 312} (2012), 87-126.
\bibitem{KU} {\sc K. Krupchyk and G. Uhlmann}, {\em Uniqueness in an inverse boundary problem for a magnetic Schrodinger operator with a bounded magnetic potential}, Comm. Math. Phys., \textbf{327} 1 (2014), 993-1009.
\bibitem{KU1}{\sc K. Krupchyk and G. Uhlmann}, {\em Stability estimates for partial data inverse problems for Schr\"odinger operators in the high frequency limit}, J. Math. Pures Appl., \textbf{126} (2019), 273-291.
\bibitem{LU}{\sc X. Li and G. Uhlmann}, {\em Inverse Problems on a Slab}, Inverse Problems and Imaging, {\bf 4} (2010), 449-462.
\bibitem{NSU} {\sc G. Nakamura, Z. Sun, G. Uhlmann}, {\em Global identifiability for an inverse problem for the Schr\"odinger equation in a magnetic field}, Math. Ann., \textbf{303} (1995), 377-388.
\bibitem{Pot2}{\sc  L. Potenciano-Machado},  {\em Optimal stability estimates for a Magnetic Schr\"odinger operator with local data}, Inverse Problems, 	\textbf{33} (2017), 095001.
\bibitem{Pot1}{\sc  L. Potenciano-Machado and A. Ruiz },  {\em Stability estimates for a Magnetic Schrodinger operator with partial data}, Inverse Probl. Imaging, \textbf{12} (2018),  1309-1342.
\bibitem {Sa1} {\sc M. Salo}, {\em Inverse problems for nonsmooth first order perturbations of the
Laplacian}, Ann. Acad. Scient. Fenn. Math. Dissertations, Vol. 139, 2004.
\bibitem {Sa2}{\sc M. Salo}, {\em Semiclassical pseudodifferential calculus and the reconstruction of a magnetic
field}, Comm. Partial Differential Equations, \textbf{31} (2006), no. 10-12, 1639-1666.
\bibitem{ST}{\sc M. Salo and L.  Tzou}, {\em Carleman estimates and inverse problems for Dirac operators}, Math.
Ann., \textbf{344} (2009), no. 1, 161-184.
\bibitem{So}{\sc Y. Soussi}, {\em Stable recovery of a non-compactly supported coefficient of a Schrödinger equation on an infinite waveguide}, to appear in Inverse Probl. Imaging, arXiv:2002.06023.
\bibitem{Suu}{\sc Z. Sun}, {\em An inverse boundary value problem for the Schr\"odinger operator with vector potentials}, Trans. Amer. Math. Soc.,  \textbf{338} No. 2 (1992), 953-969.
\bibitem{SU}{\sc J. Sylvester and G. Uhlmann}, {\em A global uniqueness theorem for an inverse boundary value problem}, Ann. of Math., {\bf 125} (1987), 153-169.
\bibitem{T}{\sc C. Tolmasky}, {\em  Exponentially growing solutions for nonsmooth first-order perturbations of the Laplacian}, SIAM J. Math. Anal., \textbf{29} (1998), no. 1, 116-133.
\bibitem{Tz}{\sc L. Tzou}, {\em Stability Estimate for the coefficients of magnetic Schr\"odinger equation from full and partial boundary measurements}, Commun. Partial Diff. Eqns., \textbf{11} (2008), 1911-1952.
\bibitem{Uh}{\sc G. 	Uhlmann}, {\em  Electrical impedance tomography and Calder\'on's problem}, Inverse problems,  \textbf{25} (2009), 123011.


\end{thebibliography}
\end{document}